\documentclass[twoside,
]{amsart}

  \usepackage[utf8]{inputenc}
  \usepackage{amsmath}
  \usepackage{amssymb}
  \usepackage{amsthm}
  \usepackage{mathrsfs}
  \usepackage{enumerate}
  \usepackage{bm}
  \usepackage{dsfont}
  \usepackage{pifont}
  \usepackage{graphicx}
  \usepackage{xmpmulti}

  \usepackage{hyperref}
  
  \usepackage[svgnames,x11names]{xcolor} 

  \hypersetup{
    colorlinks, linkcolor={Chocolate4},
    citecolor={Chocolate4}, urlcolor={DarkOliveGreen4}
  }

  \usepackage{color}
  
  \usepackage{todonotes}

	\newtheoremstyle{slanted}
	{}
	{}
	{\slshape}
	{}
	{\bfseries}
	{.}
	{ }
	{}
	
	\theoremstyle{slanted}
	\newtheorem{theo}{Theorem}[section]
	\newtheorem*{claim*}{Claim}
	\newtheorem{prop}[theo]{Proposition}

	\newtheorem{lemma}[theo]{Lemma}
	\newtheorem{definition}[theo]{Definition}
	\newtheorem{corollary}[theo]{Corollary}
	\newtheorem{remark}[theo]{Remark}
	\newtheorem{example}[theo]{Example}

	\newcommand{\egdef}{:=}
	\DeclareMathOperator{\Id}{Id}	
	\DeclareMathOperator{\Leb}{Leb}	
	\newcommand{\tend}[3][]{\xrightarrow[#2\to#3]{#1}}
	 
	\newcommand{\ind}[1]{\mathds{1}_{#1}} 
	\newcommand{\ZZ}{\mathbb{Z}}
	\renewcommand{\AA}{\mathbb{A}}

	\newcommand{\RR}{\mathbb{R}}
	
	\newcommand{\NN}{\mathbb{N}}
	\newcommand{\A}{\mathscr{A}}
	\newcommand{\B}{\mathscr{B}}

	\newcommand{\cd}{C^{d}}
	\newcommand{\td}{T^{\times d}}
	
	\newcommand{\bz}{\boldsymbol{0}}
	\newcommand{\lb}{\underline{\ell}}
	\newcommand{\nb}{\underline{n}}
	\newcommand{\kb}{\underline{k}}
	\newcommand{\jb}{\overline{j}}
	
	\newcommand{\Gam}[1]{\gamma_{#1}\bigl(\cd_{n_{\lb}}\bigr)}
	\newcommand{\Gamb}[1]{\gamma_{#1}\bigl(\overline{C}^d_{n_{\lb}}\bigr)}
	\newcommand{\Gamt}[1]{\gamma_{#1}\bigl(\tilde{C}^d_{n_{\lb}}\bigr)}
	\newcommand{\ovC}{\Omega}
	\newcommand{\ngood}{n_{\text{good}}}
	
	\newcommand{\PaP}{$\mathcal{P}a\mathcal{P}$}
	
\title{Nearly finite Chacon transformation}
\author{\'{E}lise Janvresse, Emmanuel Roy and Thierry de la Rue}
\address{\'Elise Janvresse: 
Laboratoire Amiénois de Mathématique Fondamentale et Appliquée, CNRS-UMR 
7352, Université de Picardie Jules Verne, 33 rue Saint Leu, F80039 Amiens cedex 
1,
France.}
\email{Elise.Janvresse@u-picardie.fr}
\address{Emmanuel Roy: Laboratoire Analyse, Géométrie et Applications, 
Université Paris 13 Institut Galilée,
99 avenue Jean-Baptiste Clément
F93430 Villetaneuse, France.}
\email{roy@math.univ-paris13.fr}
\address{Thierry de la Rue:
Laboratoire de Mathématiques Rapha\"el Salem,
Université de Rouen, CNRS,
Avenue de l'Université,
F76801 Saint \'Etienne du Rouvray, France.}
\email{Thierry.de-la-Rue@univ-rouen.fr}
\thanks{Research partially supported by French research group GeoSto
(CNRS-GDR3477)}

\begin{document}
\bibliographystyle{amsplain}

\maketitle
\begin{abstract}  
We construct an infinite measure preserving version of Chacon transformation, and prove that it has a property similar to Minimal Self-Joinings in finite measure: its Cartesian powers have as few invariant Radon measures as possible.
\end{abstract}



\section{Introduction}

\subsection{Motivations}

The purpose of this work is to continue the study, started in~\cite{ChaconInfinite} and~\cite{RadonMSJ}, of what the Minimal Self-Joinings (MSJ) property could be in the setting of infinite-measure preserving transformations. We want here to construct an infinite measure preserving  transformation whose Cartesian powers have as few invariant measures as possible. As in the aforementioned papers, we restrict ourselves to Radon measures (giving finite mass to compact sets), since in general there are excessively many infinite invariant measures for a given transformation (think of the sum of Dirac masses along an orbit). 

A first attempt in this direction was to consider the so-called \emph{infinite Chacon transformation} introduced in~\cite{AFS1997}. 
Indeed, the construction of this infinite measure preserving rank-one transformation is strongly inspired by the classical finite measure preserving Chacon transformation, which enjoys the MSJ property~\cite{DJRS1980}. The identification of invariant measures for Cartesian powers of the infinite Chacon transformation was the object of our previous work~\cite{ChaconInfinite}. In addition to the products of graph measures arising from powers of the transformation (see the beginning of Section~\ref{sec:main} for details), we found in the case of infinite Chacon some kind of unexpected invariant measures, the so-called \emph{weird measures}. These weird measures have marginals which are singular with respect to the original invariant measure, but it is shown in~\cite[Example~5.4]{RadonMSJ} that an appropriate convex combination of weird measures can have absolutely continuous marginals. 

We propose here another rank-one transformation, which we call the \emph{nearly finite Chacon transformation}, hereafter denoted by $T$. 
Although it preserves an infinite measure $\mu$, its construction is designed to mimic as much as possible the behaviour of the classical 
Chacon transformation, so that the phenomenon of weird measures disappears. Our main result, Theorem~\ref{thm:msj}, is the following: 
there exists a $\mu$-conull set $X_\infty$ such that, for each $d\ge1$, the ergodic $\td$-invariant Radon measures on $X_{\infty}^d$ 
are the product measure $\mu^{\otimes d}$ and products of graph measures arising from powers of~$T$. Corollary~\ref{cor:main} 
then identifies all $\td$-invariant Radon measures whose marginals are absolutely continuous with respect to $\mu$ as sums of countably 
many ergodic components which are of the form given in the theorem.

Beyond the question of the MSJ property in the infinite measure world, the example presented in this paper 
is also of crucial importance in the study of Poisson suspensions. A Poisson suspension is a finite measure preserving dynamical system constructed
from an infinite measure preserving system: a state of the space is a realization of a Poisson point process whose intensity is the infinite invariant measure, 
and each random point evolves according to the dynamics of the infinite measure preserving transformation (we refer to~\cite{Roy2007} for a complete
presentation of Poisson suspensions). 
Although of different nature, Poisson suspensions share surprising properties with another 
category of  finite measure preserving dynamical systems of probabilistic origin: Gaussian dynamical systems, which are constructed 
from finite measures on the circle.
A beautiful theory has been developped in~\cite{LemParThou00Gausselfjoin}, concerning a special class of Gaussian systems called GAGs (a 
French acronym for \emph{Gaussian systems with Gaussian self-joinings}). The keystone for the construction of a GAG system is a striking 
theorem due to Foia\c{s} and Str$\breve{\mathrm{a}}$til$\breve{\mathrm{a}}$~\cite{FoiasStratila}: if a measure supported on a 
Kronecker subset of the circle appears as the spectral measure of some ergodic stationary process, then this process
is Gaussian. The Poisson counterpart of GAG, called \PaP\ (Poisson suspension with Poisson self-joinings) is presented in~\cite{sushis},
where the construction of a \PaP\ example relies on a theorem \emph{à la Foia\c{s}-Str$\breve{\mathrm{a}}$til$\breve{\mathrm{a}}$} 
(see~\cite[Theorem~3.4]{sushis}). Roughly speaking, according to this theorem, if some ergodic point process evolves under a dynamics
directed by an infinite measure preserving transformation \emph{with special properties}, then this point process is Poissonian.
The special properties needed here are precisely those given by Corollary~\ref{cor:main}. Therefore, systems enjoying those properties
play in the theory of Poisson suspensions the same role as measures supported on Kronecker subset in the setting of Gaussian systems.

For some applications in the study of Poisson suspensions developped in~\cite{sushis}, we also need an additional property which is the existence of
a measurable law of large numbers. Proposition~\ref{prop:rational_ergodicity} shows that the nearly finite Chacon transformation satisfies 
a stronger property called rational ergodicity. 

\subsection{Roadmap of the paper}

Section~2 is devoted to the construction of the nearly finite Chacon transformation, and to first elementary results.
For pedagogical reasons, we start in Section~\ref{sec:cutandstack} by defining the nearly finite Chacon 
transformation with the cutting-and-stacking method on $ \RR_+$ equipped with the 
Lebesgue measure, as it is easier to visualize the structure of the Rokhlin towers in this setting.
Most steps of the construction are identical to construction of the classical Chacon transformation. 
There is just a fast increasing sequence $(n_\ell)$ of integers such that each $n_\ell$-th step of the construction differs from classical Chacon, which ensures that the invariant measure has infinite mass.
Then we turn in Section~\ref{sec:cantorspace} to  a more convenient (but isomorphic) model for our 
purposes, which is a transformation $T$ on a set $X$ of sequences on a countable alphabet. In Section~\ref{sec:typical},
we describe basic properties of a typical point with respect to the invariant measure $\mu$, and define the conull set $X_\infty$ referred to 
in Theorem~\ref{thm:msj}.

Section~\ref{sec:invariantmeasures} contains the main results concerning Radon measures on $X^d$ which are $\td$-invariant. 
Section~\ref{sec:basicradon} first states some basic facts about Radon measures on $X^d$. We give a criterion for such a 
measure to be $\td$-invariant (Lemma~\ref{lemma:diagonal}). We also define a notion of convergence of Radon measures (Definition~\ref{def:convergence}), 
which is specially adapted to the formulation of Hopf's ratio ergodic theorem, and give useful lemmas concerning this convergence.
In Section~\ref{sec:dissipative}, we treat the easy case of totally dissipative measures: Proposition~\ref{prop:dissipative} eliminates
 the possibility of a totally dissipative $\td$-invariant Radon measure supported on $X_\infty^d$.
In Section~\ref{sec:main}, we state our main result (Theorem~\ref{thm:msj}) and establish the bases of a proof by induction on~$d$.
At the end of Section~\ref{sec:main}, we fix once and for all a $\td$-invariant Radon measure $\sigma$ supported on $X_\infty^d$, for some $d\ge2$. 
The remainder of the paper is completely devoted to proving that either $\sigma$ is a graph measure arising from powers of $T$,
or it can be decomposed as a product of two measures on which we can apply the induction hypothesis. In section~\ref{sec:typicalpoint}, we choose once and for all a $\sigma$-typical
point $x\in X_\infty^d$, on the orbit of which we estimate the properties of $\sigma$. We also introduce in Definition~\ref{def:ncrossing} 
the central notion of \emph{$n$-crossings}, which are finite subintervals of $\ZZ$ depending on the position of the orbit of the typical point $x$ with respect to the $n$-th Rokhlin tower of the rank-one construction. 
The analysis of the structure of those $n$-crossings constitutes the core of our proof. In Section~\ref{sec:graph}, we provide a criterion 
for $\sigma$ to be a graph measure arising from powers of $T$, stated in terms of $n$-crossings (Proposition~\ref{prop:synchronized}).

Section~\ref{sec:combinatorics} is devoted to the proof of Proposition~\ref{prop:combinatorics_of_ncrossings}, which is a 
central result in the analysis of the structure of $n$-crossings. Section~\ref{sec:hierarchy} describes a hierarchy of abstract subsets
of $\ZZ$ and provides a lemma on the combinatorics of subsets in this hierarchy (Lemma~\ref{lemma:combinatorics}). 
Then Section~\ref{sec:application} explains how to apply this lemma to the structure of $n$-crossings. Section~\ref{sec:edge}
provides a useful corollary of Proposition~\ref{prop:combinatorics_of_ncrossings} in terms of the measure $\sigma$.

Section~\ref{sec:convergence} is devoted to the study of the convergence of \emph{empirical measures}, which are finite sums
of Dirac masses on points on the orbit of $x$, corresponding to finite subsets of $\ZZ$.
We provide two criteria, Proposition~\ref{prop:convergence_empirical_measures_interval}
and Proposition~\ref{prop:convergence_empirical_measures_pieces}, for a sequence of such empirical measures to converge to $\sigma$.

In Section~\ref{sec:twist} we present the main tool used to decompose $\sigma$ as a product
measure. We introduce the notion of \emph{twisting transformation} (definition~\ref{def:twist}), which is simply a transformation of $X^d$ acting as $T$ on some coordinates, and as $\Id$ on others.
Based on a theorem from~\cite{ChaconInfinite}, Proposition~\ref{prop:product} shows that, if $\sigma$ is invariant by such a twisting transformation, then $\sigma$ decomposes as a product measure to which we can apply the induction hypothesis. Then Proposition~\ref{prop:twist} provides a criterion for $\sigma$ to be invariant by some twisting transformation.

All the preceding tools are used in Section~\ref{sec:end}, where the proof of Theorem~\ref{thm:msj} is completed. If the criterion given by Proposition~\ref{prop:synchronized} for $\sigma$ to be a graph measure fails, then 
for infinitely many integers $n$ there exists some $n$-crossing, not too far from 0,
with some special property. 
We treat several cases according to the positions of these integers with respect to the sequence $(n_\ell)$. With the help of Propositions~\ref{prop:twist} and \ref{prop:product}, we show that in all cases $\sigma$ decomposes as a product of two measures to which the induction hypothesis applies. 

Section~\ref{sec:rational_ergodicity}, which can be read independently, is devoted to the proof of the rational ergodicity of the nearly finite Chacon transformation.

\section{Construction of the nearly finite Chacon transformation}
\label{sec:construction}

\subsection{Cutting-and-stacking construction on $\RR_+$}
\label{sec:cutandstack}
As previously explained, this transformation is designed to mimic the classical 
finite measure preserving Chacon transformation as much as possible, yet it must 
preserve an infinite measure.
The construction will make use of two predefined increasing sequences of 
integers: $1\ll n_1 \ll n_2 \ll \cdots \ll n_\ell \ll \cdots$ and $\ell_0:=1\ll 
\ell_1\ll \ell_2\ll\cdots\ll\ell_k\ll\cdots$, satisfying some growth conditions 
to be precised later (see below conditions~\eqref{eq:condition_lk} 
and~\eqref{eq:condition_nl}). For each $\ell\ge1$, there exists a unique integer $k\ge0$
such that $\ell_k\le \ell <\ell_{k+1}$, and we denote this integer by $k(\ell)$.

In the first step we consider the interval $[0,1)$, which is cut into three 
subintervals of equal length. We take the extra interval $[1,4/3)$ and stack it 
above the middle piece. Then we stack all these intervals left under right, 
getting a tower of height $h_1\egdef4$. The transformation $T$ maps each point 
to the point exactly above it in the tower. At this step $T$ is yet undefined on 
the top  of the tower.

After step $n$ we have a tower of height $h_n$, called tower~$n$, the levels of 
which are intervals of length $1/3^n$. These intervals are closed to the left 
and open to the right. At step~$(n+1)$, tower~$n$ is cut into three subcolumns 
of equal width. 
If $n\notin\{n_\ell:\ell\ge1\}$, we do as in the standard finite measure 
preserving Chacon transformation: we add an extra interval of length $1/3^{n+1}$ 
above the middle subcolumn, and we  stack the three subcolumns left under right 
to get tower~$n+1$ of height $h_{n+1}=3h_n+1$.
If $n=n_\ell$ for some $\ell$, 
we add $h_{n-k(\ell)}$ extra intervals above each of the three 
subcolumns, and a further extra interval above the second subcolumn.  Then we 
stack the three subcolumns left under right and get tower~$n+1$ of height 
$h_{n+1}=3h_n + 3h_{n-k(\ell)} + 1$. (See Figure~\ref{fig:construction}.)

At each step, we pick the extra intervals successively by taking the leftmost 
interval of desired length in the unused part of $\RR_+$. 
Extra intervals used at step $n+1$ are called \emph{$(n+1)$-spacers}. 

We want the Lebesgue measure of tower~$n$ to increase to infinity, which is easily 
satisfied provided the sequence $\ell_k$ grows sufficiently fast. 
Indeed, for each $n\ge1$ we have $ h_{n+1} \le 6h_n+1 \le 7h_n$, whence 
${h_n}/{h_{n+1}}\ge 1/7$. It follows that for each $k\ge0$ and each $\ell_k\le 
\ell<\ell_{k+1}$, 
\[  \Leb(\text{tower }n_\ell+1)\ge \Leb(\text{tower 
}n_\ell)\left(1+\frac{h_{n_\ell-k}}{h_{n_\ell}}\right) \ge 
\left(1+7^{-k}\right). \]
Therefore it is enough for example to assume that for each $k\ge0$,
\begin{equation}
\label{eq:condition_lk}
   \left(1+7^{-k}\right)^{\ell_{k+1}-\ell_k}\ge 2.
\end{equation}
Under this assumption, we 
get at the end a rank-one transformation defined on $\RR_+$  which preserves the 
Lebesgue measure. 

We will also assume that for each $\ell$, 
\begin{equation}
\label{eq:condition_nl}
  n_\ell >  n_{(\ell-1)} + 2\ell.
\end{equation}

\begin{figure}[htp]
  \centering
  \includegraphics[width=10cm]{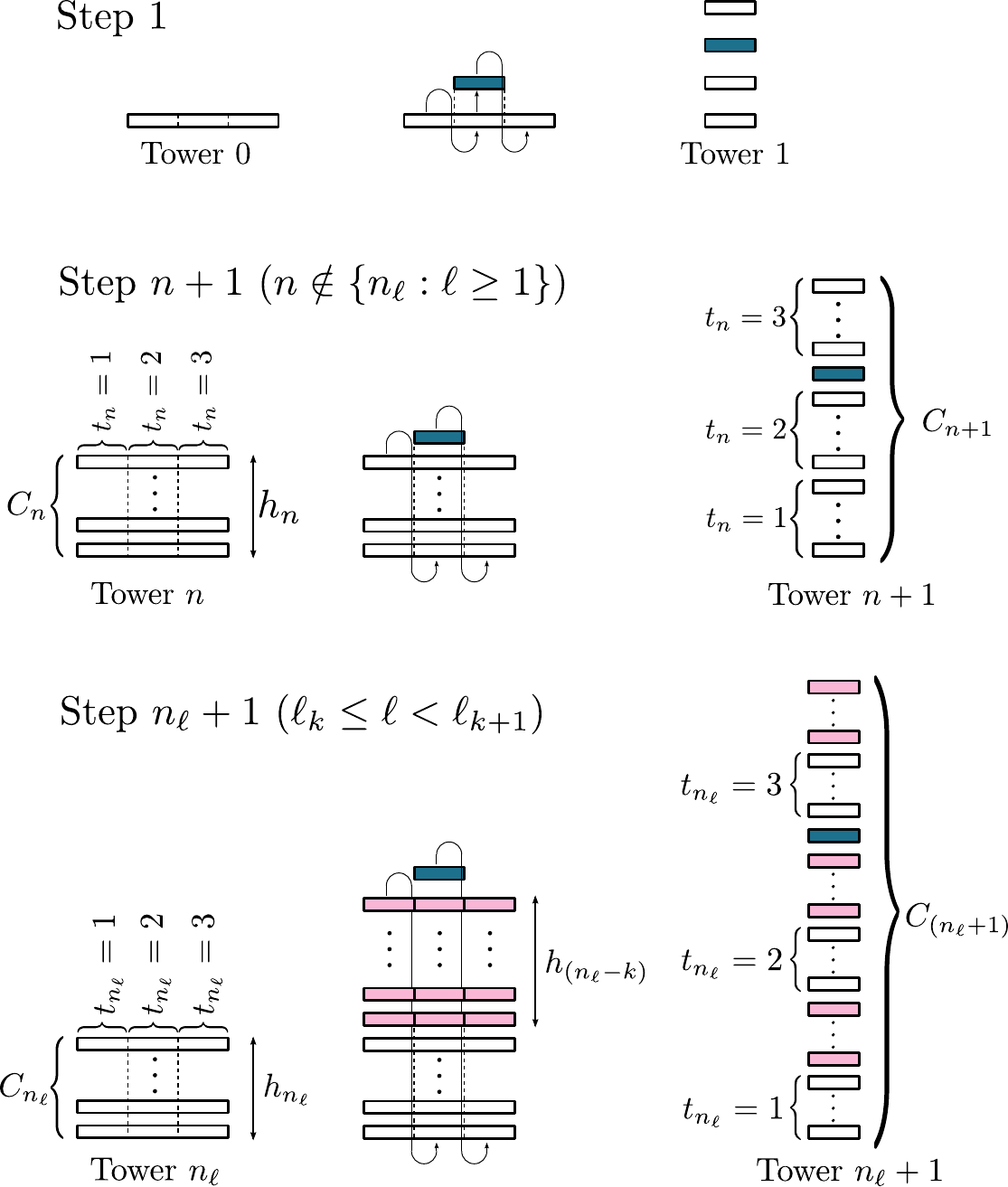}
  \caption{Construction of the nearly finite Chacon transformation by cutting 
and stacking.}
  \label{fig:construction}
\end{figure}

\subsection{Construction on a set of sequences}
\label{sec:cantorspace}
For technical reasons, it will be more convenient to consider a model of the 
nearly finite Chacon transformation in which the ambient space is a totally 
disconnected non compact metric space, and each level of each tower is a compact 
clopen set. 

Consider the countable alphabet $\AA\egdef\{\ast\}\cup\NN$. To each $t\in\RR_+$, 
we associate the sequence $\varphi(t)={\bigl(j_n(t)\bigr)}_{n\ge0}\in\AA^\NN$ 
defined by
\[
j_n(t)\egdef  \begin{cases}
                \ast &\text{ if }t\notin\text{tower }n,\\
                j &\text{ if }t\text{ is in level $j$ of tower $n$ }(0\le 
j<h_n).
              \end{cases}
\]
By condition~\eqref{eq:condition_lk}, $\RR_+=\bigcup_n \text{tower }n$, and for 
each $n$, $\text{tower }n\subset \text{tower }n+1$. Hence for each $t\in\RR_+$,
\begin{equation*}
  \label{eq:sequences-C1}
  \exists n_0\ge0: \forall n<n_0,\ j_n(t)=\ast, \text{ and } \forall n\ge n_0, 
j_n(t)\in\{0,\ldots,h_n-1\}.
\end{equation*}
Moreover, each level of tower $n+1$ is either completely outside tower~$n$ or 
completely inside a single level of tower~$n$. Let us introduce, for each 
$n\ge1$, the map 
$p_n:\{0,\ldots, h_{n+1}-1,\ast\}\to \{0,\ldots, h_{n}-1,\ast\}$ defined by
\begin{itemize}
  \item $p_n(\ast)\egdef\ast$,
  \item $\forall j\in\{0,\ldots, h_{n+1}-1\},\ p_n(j)\egdef \ast$ if level $j$  
of tower $n+1$ is  completely outside tower~$n$, and $p_n(j)\egdef 
j'\in\{0,\ldots, h_{n}-1\}$ if 
  level~$j$  of tower~$n+1$ is  completely inside level $j'$ of tower~$n$.
\end{itemize}
Then the sequence ${\bigl(j_n(t)\bigr)}_{n\ge0}$ satisfies the following 
compatibility condition.
\begin{equation*}
  \label{eq:sequences-C2}
  \forall n\ge0,\  j_n(t)=p_n\bigl(j_{n+1}(t)\bigr).
 \end{equation*}
In particular, $j_n(t)$ completely determines $j_m(t)$ for each $0\le m\le n$. 
We also observe that $p_n$ satisfies the following property:
\begin{equation}
\label{eq:prop_pn}
  \text{If }p_n(j)\in\{0,\ldots,h_n-2\}\text{, then 
}j\in\{0,\ldots,h_{n+1}-2\}\text{ and }p_n(j+1)=p_n(j)+1.
\end{equation}

Now we can define our space $X$, to which belongs $\varphi(t)$ for each 
$t\in\RR_+$:
\[
  X\egdef \left\{(j_n)_{n\ge0}\in\AA^{\NN}: \forall n\ge0,\  
j_n=p_n(j_{n+1})\text{ and }\exists n_0,\  j_{n_0}\neq\ast \right\}.
\]
$X$ inherits its topology from the product topology of $\AA^{\NN}$. In 
particular it is a totally disconnected metrizable space, but it is not compact 
(in fact $X$ is 
not closed in $\AA^{\NN}$, as the infinite sequence $(\ast,\ast,\ldots)$ is in 
$\overline{X}\setminus X$).

For each $n\ge0$ and each $x\in X$, we denote by $j_n(x)$ the $n$-th coordinate of $x$.
For each $j\in\{0,\ldots,h_n-1\}$, we define the subset of $X$ 
\[L_n^j\egdef\{x\in X: j_n(x)=j\}.\] 
Then $L_n^j$ is compact and clopen 
in $X$. Moreover the family of sets $(L_n^j)$
form a basis of the topology on~$X$.

To define the transformation $T$ on $X$, we need the following easy lemma.
\begin{lemma}
\label{lemma:defT}
  For each $x=(j_n)_{n\ge0}\in X$, there exists $\overline{n}$ such that, for 
each $n\ge \overline{n}$, $j_n \in \{0,\ldots,h_n-2\}$.
\end{lemma}
\begin{proof}
  Remember that at each step $n_{\ell}+1$, some spacers are added on the last 
subcolumn of tower~$n_\ell$. 
  Hence, $j_{n_\ell+1}=h_{n_\ell+1}-1$ implies $j_{n_\ell}= \ast$. Now take 
$\ell$ large enough so that $j_{n_\ell}\neq \ast$.
  Then $j_{n_\ell+1}<h_{n_\ell+1}-1$, and~\eqref{eq:prop_pn} shows by an 
immediate induction that $j_m < h_m-1$ for each $m\ge n_\ell+1$.
\end{proof}

We define the measurable transformation $T:X\to X$ as follows: for 
$x=(j_n)_{n\ge0}\in X$, we consider the smallest integer $\overline{n}$ 
satisfying the property given in Lemma~\ref{lemma:defT}.
Then we set $T(x)\egdef (j'_n)_{n\ge0}$, where $j'_n\egdef j_n+1$ if 
$n\ge\overline{n}$, and the finite sequence 
$(j'_1,j'_2,\ldots,j'_{\overline{n}-1})$ is determined by the value of 
$j'_{\overline{n}}$ and the compatibility conditions $j'_n=p_n(j'_{n+1})$, $1\le 
n<\overline{n}$. 
Note that $T$ is one-to-one, and $T(X)=X\setminus\{\bz\}$, where 
$\bz\egdef(0,0,\ldots)$. 

For each $n\ge1$ and each $0\le j<h_n-1$, $T(L_n^j)=L_n^{j+1}$, hence 
$(L_n^0,\ldots,L_n^{h_n-1})$ is a Rokhlin tower for $T$. By construction, the 
family of Rokhlin towers we get in this way 
has the same structure as the family of Rokhlin towers we constructed by 
cutting-and-stacking on $\RR_+$. 
From now on, ``tower~$n$" will rather designate the Rokhlin tower 
$(L_n^0,\ldots,L_n^{h_n-1})$ in $X$. 
The main advantage that we get compared to the construction on $\RR_+$ is the 
following elementary fact.
\begin{remark}
  \label{rem:nonemptyintersection}
  If $\left(L_n^{j_n}\right)_{n\ge \overline{n}}$ is a sequence of levels in the 
successive Rokhlin towers, such that $L_{n+1}^{j_{n+1}}$ is always included in 
$L_n^{j_n}$ (equivalently, $j_n=p_n(j_{n+1})$), then 
$\bigcap_n L_n^{j_n}$ is always a singleton
\end{remark}
 (Note that such an intersection can be empty in the construction on $\RR_+$.)

Let $\mu$ be the pushforward of the Lebesgue measure on $\RR_+$ by $\varphi$. 
Then $\mu$ is an infinite, $\sigma$-finite and $T$-invariant measure on $X$, and 
it satisfies 
\[
  \forall n\ge0,\ \forall j\in\{0,\ldots,h_n-1\},\ \mu(L_n^j)=3^{-n}.
\]

\subsubsection*{Additional notations}
For each $n\ge0$, we denote by $C_n$ the subset of $X$ formed by the union of 
all the levels of tower~$n$. Note that for each $n\ge0$,
$C_n\subset C_{n+1}$, and that $X=\bigcup_{n\ge0}C_n$.
For $x\in C_n$, note that $j_n(x)$ indicates the level of tower $n$ to which $x$ 
belongs.


We also define a function $t_n$ on $C_n$, taking values in $\{1,2,3\}$, which 
indicates for each point whether it belongs to the first, the second, or the 
third subcolumn of tower~$n$. We thus have for $x\in C_n$ and $n\notin\{n_\ell:\ell\ge1\}$
\begin{equation} \label{eq:induction_for_j_n}
  j_{n+1}(x) = \begin{cases}
                 j_n(x) & \text{ if }t_n(x)=1,\\
                 j_n(x)+h_n & \text{ if }t_n(x)=2,\\
                 j_n(x)+2h_n+1 & \text{ if }t_n(x)=3.\\
               \end{cases}
\end{equation}
In the case where $n=n_\ell$ for some $\ell\ge1$, we have to replace $h_n$ by $h_{n_\ell}+h_{n_\ell-k(\ell)}$ 
in the above formula.
In particular, we always have
\begin{equation}
  \label{eq:s_n}
  j_{n+1}(x) \ge j_n(x).
\end{equation}

Consider two integers $0\le n < m$. By construction, tower $n$ is subdivised 
into $3^{m-n}$ subcolumns which appear as 
bundles of $h_n$ consecutive levels in tower $m$: we call them \emph{occurrences 
of tower $n$ inside tower $m$}. These occurrences are naturally ordered, from 
bottom to 
top of tower $m$. For a point $x$ in tower $n$, the precise occurrence of tower 
$n$ inside tower $m$ to which $x$ belongs is determined by the sequence 
$t_n(x),t_{n+1}(x),\ldots,t_{m-1}(x)$.
For example, $x$ belongs to the last occurrence of tower $n$ inside tower $m$ if 
and only if $t_n(x)=t_{n+1}(x)=\cdots=t_{m-1}(x)=3$.

\begin{remark}
\label{remark:01spacer}
  Observe that for each $\ell\ge2$ and each $n_{(\ell-1)}+1\le n\le n_\ell - 1$, 
there is 0 or 1 spacer between two consecutive occurrences of tower $n$ inside 
tower $n_\ell$.
\end{remark}

\subsection{Behaviour of $\mu$-typical points}

\label{sec:typical}

\begin{lemma}
  \label{lemma:Xinfty}
 There exists a $\mu$-conull subset $X_\infty$ of $X$ satisfying: for each $x\in 
X_\infty$, there exists an integer $\ell(x)$ such that, for all $\ell\ge 
\ell(x)$, for each 
 $n_{(\ell-1)} \le n \le n_\ell -\ell$, $x\in C_n$ but $x$ is neither in the 
first hundred nor in the last hundred occurrences of tower $n$ inside tower $n_\ell$.
\end{lemma}

\begin{proof}
If we consider $x$ as a random point chosen according to the normalized 
$\mu$-measure on $C_n$, then the random variables 
$t_n(x),t_{n+1}(x),\ldots,t_{m-1}(x)$ are i.i.d. and uniformly distributed in 
$\{1,2,3\}$. Hence the probability that $x$ belongs to some specified occurrence 
of tower $n$ inside tower $m$ is $1/3^{m-n}$. 

Since the series $\sum 1/3^\ell$ converges, by Borel Cantelli there exists a 
subset $X_n$ of full $\mu$-measure inside $C_n$ such that, for each $x\in X_n$, 
there is only a finite number of integers $\ell$ such that $x$ belongs to the 
first hundred or to the last hundred occurrences of tower $(n_\ell-\ell)$ inside tower $n_\ell$. 

Setting 
\[
  X_\infty \egdef X\setminus \bigcup_n \left(C_n\setminus X_n\right),
\]
we get a conull subset of $X$, and for each $x\in X_\infty$, there exists an 
integer $\ell(x)$ such that, for all $\ell\ge \ell(x)$, $x\in 
C_{n_{(\ell-1)}}\subset C_{n_\ell-\ell}$, but $x$ is neither in the first hundred nor 
in the last hundred occurrence of tower $(n_\ell-\ell)$ inside tower $n_\ell$. If 
$n_{(\ell-1)} \le n \le n_\ell -\ell$, the first (respectively last) hundred occurrences 
of tower $n$ inside tower $n_\ell$ are included in the first (respectively last) hundred
occurrences of tower $(n_\ell-\ell)$ inside tower $n_\ell$, and this concludes 
the proof.
\end{proof}

\begin{remark} 
\label{remark:first_level}
  In particular, for each $x\in X_\infty$, if $n>n_{\ell(x)}$, then $x$ does not 
belong to the first level of tower~$n$. And since $\bz$ is in the first level of tower~$n$ for each $n$,
we have $\bz\notin X_\infty$.
\end{remark}

\begin{remark}
\label{remark:1.6}
  As $n_{(\ell-1)}+\ell < n_\ell-\ell$ by~\eqref{eq:condition_nl},  we may also assume that  for each $x\in X_\infty$
  and each $\ell\ge \ell(x)$, $x$ is neither in the first hundred nor in the last hundred occurrences of tower $n_{(\ell-1)}$ inside tower $n_\ell-\ell$.
\end{remark}

\section{Invariant Radon measures for Cartesian powers of the nearly finite Chacon transformation}
\label{sec:invariantmeasures}
We fix a natural integer $d\ge1$, and we study the action of the Cartesian power 
$\td$ on $X^d$. Recall that a measure $\sigma$ on $X^d$ is a \emph{Radon 
measure} if it is finite on each compact subset of $X^d$ (equivalently, if 
$\sigma(\cd_n)<\infty$ for each $n$). In particular, a Radon measure on $X^d$ is 
$\sigma$-finite (but the converse is not true).

Our purpose is to describe all Radon measures on $X^d$ which are $\td$-invariant 
and whose marginals are absolutely continuous with respect to $\mu$. 

\subsection{Basic facts about Radon measures on $X^d$}
\label{sec:basicradon}
We call \emph{$n$-box} a subset of $\cd_n$ which is a Cartesian product 
$L_n^{j_1}\times \cdots\times L_n^{j_d}$, where each $L_n^{j_i}$ is a level of 
the Rokhlin tower~$C_n$.
A \emph{box} is a subset which is an $n$-box for some $n\ge0$.
The family of all  boxes  form a basis of compact clopen sets of the topology of 
$X^d$. 

We consider the following ring of subsets of $X^d$
\[
    \mathscr{R}\egdef\{B\subset X^d:\ \exists n\ge0,\ B\text{ is a finite union 
of $n$-boxes}\}.
\]

\begin{prop}
\label{prop:caratheodory-extension}
  Any finitely additive functional $\sigma:\mathscr{R}\to\RR_+$ can be extended 
to a unique measure on the Borel $\sigma$-algebra $\B(X^d)$, which is Radon.
\end{prop}
\begin{proof}
  Using Theorems~F p.~39 and A p.~54 (Caratheodory's extension theorem) 
in~\cite{Halmos}, we only have to prove that, if $(R_k)_{k\ge 1}$ is a 
decreasing sequence in $\mathscr{R}$ such that 
  $  \lim_{k\to\infty}\downarrow \sigma(R_k) > 0$,
  then $\bigcap_k R_k\neq \emptyset$.
  But this is obvious since, under this assumption, each $R_k$ is a compact 
nonempty set.
\end{proof}

In particular, to define a Radon measure $\sigma$ on $X^d$, we only have to 
define $\sigma(B)$ for each box $B$, with the compatibility condition that,
if $B$ is an $n$-box for some $n\ge0$, then $\sigma(B)=\sum_{B'\subset B} 
\sigma(B')$, where the sum ranges over the $3^d$ $(n+1)$-boxes which are 
contained in $B$.

We call \emph{$n$-diagonal} a Rokhlin tower for $\td$ which is of the form 
\[ \bigl(B, \td(B), \ldots, (\td)^{r-1}(B)\bigr),\] 
where each $(\td)^j(B)$ is an $n$-box, and which is maximal in the following 
sense: $B$ has one projection which is the bottom level $L_n^0$ of tower~$n$, 
$(\td)^{r-1}B$ has one projection 
which is the top level $L_n^{h_n-1}$ of tower~$n$, and the projections of each 
$(\td)^jB$, $1\le j\le r-2$ are neither the bottom level nor the top level of 
tower~$n$. (See Figure~\ref{fig:diago})


\begin{figure}[htp]
  \centering
  \includegraphics[width=11cm]{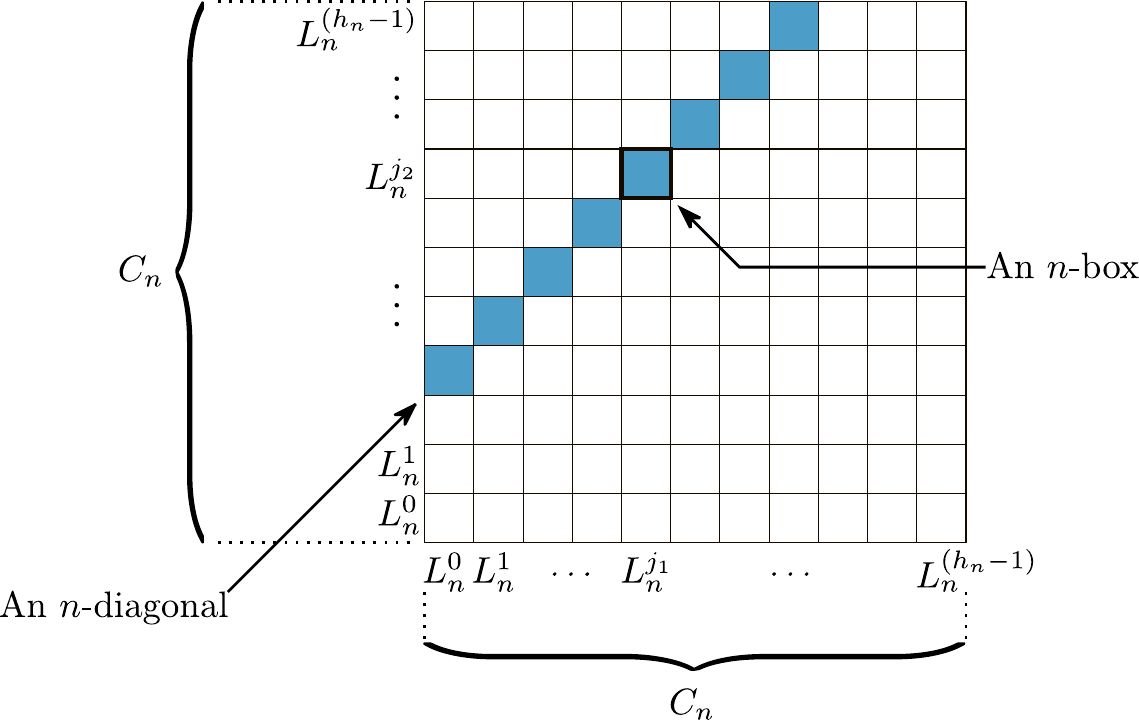}
  \caption{An $n$-diagonal inside $C_n^d$ (here $d=2$)}
  \label{fig:diago}
\end{figure}

\begin{lemma}
 \label{lemma:diagonal}
 Set $X_{\bz}^d\egdef\{x=(x_1,\ldots,x_d)\in X^d: \exists i, x_i=\bz\}$. 
 Let $\sigma$ be a Radon measure on $X^d$. Then $\sigma$ is $\td$-invariant if 
and only if the following two conditions hold:
 \begin{enumerate}
   \item $\sigma\left(X_{\bz}^d\right)=0$.
   \item for each $n$, all the $n$-boxes lying on an $n$-diagonal always have the same 
measure.
 \end{enumerate}
\end{lemma}
\begin{proof}
  Assume first that $\sigma$ is a $\td$-invariant Radon measure on $X^d$. 
Recalling that $\bz$ has no preimage by $T$, we see that 
$(\td)^{-1}(X_{\bz}^d)=\emptyset$, 
  whence $\sigma\left(X_{\bz}^d\right)=0$. Moreover, since $n$-boxes on an 
$n$-diagonal are levels of a $\td$-Rokhlin tower, the second condition obviously 
holds. 
  Reciprocally, assume that the two conditions given in the statement of the 
lemma hold. For each $n$, let 
  $\ovC_n$ be the subset of $\cd_n$ constituted of all $n$-boxes of the form $L_n^{j_1}\times\cdots\times L_n^{j_d}$, where for each $i$ $j_i\neq 0$. 
  Then the second condition implies that 
  $\sigma$ and $(\td)_*(\sigma)$ coincide on $\ovC_n$  for each $n$. 
But
  \[
    \bigcup_{n\ge0} \ovC_n = X\setminus X_{\bz}^d.
  \]
  On the other hand, we have 
$(\td)_*(\sigma)(X_{\bz}^d)=\sigma\bigl((\td)^{-1}
(X_{\bz}^d)\bigr)=\sigma(\emptyset)=0$. With the first condition we see that 
  $\sigma$ and $(\td)_*(\sigma)$ also coincide on $X_{\bz}^d$, hence they are equal.
\end{proof}

\begin{definition}[Convergence of Radon measures on $X^d$]
\label{def:convergence}
  We say that a sequence $(\sigma_k)$ of Radon measures on $X^d$ converges to 
the nonzero Radon measure $\sigma$ if, 
  for each $n$ large enough so that $\sigma(\cd_n)>0$, we have
  \begin{itemize}
    \item $\sigma_k(\cd_n)>0$ for all large enough $k$,
    \item for each $n$-box $B$,  
    \[
      \frac{\sigma_k(B)}{\sigma_k(\cd_n)} \tend{k}{\infty} 
\frac{\sigma(B)}{\sigma(\cd_n)}.
    \]
  \end{itemize}
\end{definition}

Observe that, when a sequence of Radon measures converges in the above sense, 
then its limit is unique up to a multiplicative positive constant. Observe also 
that the convergence 
is unchanged if we multiply each measure $\sigma_k$ by a positive real number 
(which may vary with $k$).


\begin{remark}
\label{rem:mn}
 If the sequence $(\sigma_k)$ of Radon measures on $X^d$ converges to the 
nonzero Radon measure $\sigma$, then for $n$ such that $\sigma(\cd_n)>0$ and for 
each $m\ge n$,
 for each $m$-box $B\subset\cd_n$, we also have  
    \[
      \frac{\sigma_k(B)}{\sigma_k(\cd_n)} \tend{k}{\infty} 
\frac{\sigma(B)}{\sigma(\cd_n)}.
    \]
 Consequently, the above holds also when $B\in\mathscr{R}$ is included in 
$\cd_n$.
\end{remark}

Indeed, as $\cd_n$ is a finite union of $m$-boxes, we have 
    \[
      \frac{\sigma_k(\cd_n)}{\sigma_k(\cd_m)} \tend{k}{\infty} 
\frac{\sigma(\cd_n)}{\sigma(\cd_m)}.
    \]
Then we can write, for an $m$-box $B\subset\cd_n$,
\[
 \frac{\sigma_k(B)}{\sigma_k(\cd_n)} = \frac{\sigma_k(B)}{\sigma_k(\cd_m)}  
\frac{\sigma_k(\cd_m)}{\sigma_k(\cd_n)}
 \tend{k}{\infty} \frac{\sigma(B)}{\sigma(\cd_m)}  
\frac{\sigma(\cd_m)}{\sigma(\cd_n)} = \frac{\sigma(B)}{\sigma(\cd_n)}.
\]

\begin{prop}
  \label{prop:compacity}
  Let $(\sigma_k)$ be a sequence of Radon measures on $X^d$, and assume that 
there exists some $\underline{n}\ge 0$ satisfying
  \begin{itemize}
   \item $\sigma_k(\cd_{\underline{n}})>0$ for
  all large enough $k$,
  \item for each $n>\underline{n}$, the sequence 
$\left(\sigma_k(\cd_n)/\sigma_k(\cd_{\underline{n}})\right)_{k}$ is bounded.
  \end{itemize}
Then there is a subsequence $(k_j)$ and a nonzero Radon measure $\sigma$ on 
$X^d$ such that 
  $(\sigma_{k_j})$ converges to $\sigma$.
\end{prop}
\begin{proof}
  Multiplying each $\sigma_k$ by a positive real number if necessary, we may assume 
that for all large enough $k$, 
  $\sigma_k(\cd_{\underline{n}})=1$. Then the second assumption ensures that for 
each box $B$, the sequence $\left(\sigma_k(B)\right)_k$ is bounded.
  By a standard diagonal procedure, we can find a subsequence  $(k_j)$ such that 
for each box $B$, $\sigma_{k_j}(B)$ has a limit which we denote by $\sigma(B)$.
  Then $\sigma$ defines a finitely additive functional on the ring $\mathscr{R}$ 
of finite unions of boxes, with values in $\RR_+$. By 
Proposition~\ref{prop:caratheodory-extension},
  $\sigma$ can be extended to a Radon measure on $\B(X^d)$, which is nonzero 
since $\sigma(\cd_{\underline{n}})=1$. And we obviously have the convergence of 
$(\sigma_{k_j})$ 
  to $\sigma$.
\end{proof}

\begin{prop}
  \label{prop:absolute_continuity}
  Let $(\sigma_k)$ and $(\gamma_k)$ be two sequences of Radon measures on $X^d$, 
and assume there exist two nonzero Radon measures $\sigma$ and $\gamma$, 
  an integer $\underline{n}\ge1$ and a real number $\theta>0$
  such that
  \begin{itemize}
   \item $\sigma_k\tend{k}{\infty}\sigma$,
   \item $\gamma_k\tend{k}{\infty}\gamma$,
   \item $\forall k$, $\gamma_k\le\sigma_k$
   \item $\forall n\ge \underline{n}$, for all large enough $k$ (depending on 
$n$), $\gamma_k(\cd_n)\ge \theta \sigma_k(\cd_n)$.
  \end{itemize}
  Then $\gamma\ll\sigma$.
\end{prop}
  
\begin{proof}
 Let $m\ge n\ge \underline{n}$, and let $B$ be an $m$-box. For all large enough 
$k$, we have by assumption
 \[
  \frac{\gamma_k(B)}{\gamma_k(\cd_n)} \le 
\frac{\sigma_k(B)}{\theta\sigma_k(\cd_n)}.
 \]
 But by Remark~\ref{rem:mn}, we have 
 \[
  \frac{\gamma_k(B)}{\gamma_k(\cd_n)} \tend{k}{\infty} 
\frac{\gamma(B)}{\gamma(\cd_n)},\text{ and }
  \frac{\sigma_k(B)}{\sigma_k(\cd_n)} \tend{k}{\infty} 
\frac{\sigma(B)}{\sigma(\cd_n)}.
 \]
It follows that 
\[
 \frac{\gamma(B)}{\gamma(\cd_n)} \le \frac{\sigma(B)}{\theta\sigma(\cd_n)}.
\]
The above inequality extends to each $B\in\mathscr{R}$ contained in $\cd_n$, 
and then to each $B\in\mathscr{B}(X)$ contained in 
$\cd_n$. In particular, if $B\subset \cd_n$ is Borel measurable and satisfies 
$\sigma(B)=0$, then we also have $\gamma(B)=0$. 
And since $X=\bigcup_n \cd_n$, this concludes the proof.
\end{proof}

\begin{remark}
  For each $\lb\ge1$, the definition of the ring $\mathscr{R}$ is unchanged if we consider only the finite unions of  $n_\ell$-boxes, for some $\ell\ge\lb$. Hence in Propositions~\ref{prop:compacity} and \ref{prop:absolute_continuity}, it is 
  enough for the conclusions to hold that the assumptions be verified only when $n\in\{n_\ell:\ell\ge\lb\}$.
\end{remark}

\subsection{Dissipative case}
\label{sec:dissipative}
\begin{lemma}
\label{lemma:dissipative_case}
  For each $x=(x_1,\ldots,x_d)\in X_\infty^d$, for $\ell > 
\max\{\ell(x_i):i=1,\ldots,d\}$, we have
  \[
   \#\{j\ge0: (\td)^j(x)\in C_{n_\ell + 1}^d\} = \infty. 
  \]
\end{lemma}
\begin{proof}
  If $\ell> \max\{\ell(x_i):i=1,\ldots,d\}$, we know by Lemma~\ref{lemma:Xinfty} 
that each coordinate $x_i$ is in $C_{n_\ell + 1}$, 
but is not in the last occurrence of tower $(n_\ell+1)$ inside tower 
$n_{(\ell+1)}$. Moreover by Remark~\ref{remark:first_level}, $x_i$ is not in the 
first level of tower $(n_\ell+1)$. The next occurrence of tower $(n_\ell+1)$ 
inside tower $n_{(\ell+1)}$ appears after 0 or 1 spacer by 
Remark~\ref{remark:01spacer}. As the height of tower $(n_\ell+1)$ is 
$h_{(n_\ell+1)}$, $T^{h_{(n_\ell+1)}}(x_i)$ is either in the same level of tower 
$(n_\ell+1)$ as $x_i$, or in the level immediately below. 
Thus $T^{h_{(n_\ell+1)}}(x_i)\in C_{n_\ell + 1}$. 
  But the same applies to any $\ell'\ge \ell$, and we get that 
$T^{h_{(n_{\ell'}+1)}}(x_i)$  is either in the same level of tower 
$(n_{\ell'}+1)$ as $x_i$, or in the level immediately below. Since these two 
levels are both included in $C_{n_\ell + 1}$ we get that 
$T^{h_{(n_{\ell'}+1)}}(x_i)\in C_{n_\ell + 1}$.
\end{proof}

\begin{prop}
  \label{prop:dissipative}
  There is no Radon, $\td$-invariant and totally dissipative measure for which $X_\infty^d$ is a conull set. In particular, there is no Radon, $\td$-invariant and totally dissipative measure whose 
marginals are absolutely continuous with respect to $\mu$.
\end{prop}

\begin{proof}
  Suppose that $\sigma$ is such a measure. Let $W$ be a wandering set for 
$\sigma$, with 
   \[ \sigma\left(X^d\setminus\bigcup_{j\in\ZZ}(\td)^jW\right)=0.
      \]    
 As $X_\infty^d$ is a conull set, we may assume that 
  $W\subset X_\infty^d$. 
By the previous lemma, $W=\bigcup_{n}W_n$, where 
\[ W_n\egdef \Bigl\{ 
  x\in W: \#\{j\ge0: (\td)^j(x)\in C_{n}^d\} = \infty
\Bigr\}.
\]
Hence there exists some $n$ with $\sigma(W_n)>0$. The ergodic decomposition of 
$\sigma$ writes
  \[
    \sigma = \int_W \left(\sum_{j\in\ZZ}\delta_{(\td)^j(x)}\right) \, 
d\sigma(x),
  \]
  so we get $\sigma(C_n^d)=\infty$, which contradicts the fact that $\sigma$ is 
Radon.
\end{proof}

\subsection{Main result}
\label{sec:main}
An obvious example of a $\td$-invariant Radon measure on $X^d$ is the product measure $\mu^{\otimes d}$.
Another example is what we call a 
\emph{graph measure arising from powers of $T$}: this is a measure $\sigma$ of the form
\begin{equation}
  \label{eq:graph_measure}
  \sigma(A_1\times\cdots\times A_d) = \alpha\mu(A_1\cap T^{-e_2}A_2\cap\cdots\cap T^{-e_d}A_d),
\end{equation}
for some integers $e_2,\ldots,e_d$ and some fixed positive real number $\alpha$. Such a measure is concentrated 
on the subset 
\[
  \left\{(x_1,\ldots,x_d)\in X^d: x_i=T^{e_i}x_1 \text{ for all }i=2,\ldots,d\right\}.
\]

\begin{theo}
  \label{thm:msj}
  For each $d\ge1$, the infinite measure preserving dynamical system $(X^d, 
\mu^{\otimes d},\td)$ is conservative ergodic.
  
  Moreover, if $\sigma$ is a nonzero, Radon, $\td$-invariant and ergodic measure 
on $X^d$, such that
\begin{equation}
  \label{eq:weaker}
  \sigma \left(X^d \setminus X_\infty^d\right)=0,
\end{equation}  
then 
there exists a partition of $\{1,\ldots,d\}$ into $r$ subsets $D_1,\ldots,D_r$, 
such that $\sigma=\sigma^{D_1}\otimes \cdots\otimes \sigma^{D_r}$, where 
$\sigma^{D_j}$ is a graph measure on $X^{D_j}$ arising from powers of~$T$.  
\end{theo}

\begin{corollary}
\label{cor:main}  If $\sigma$ is a nonzero, Radon, $\td$-invariant  measure 
on $X^d$, whose marginals are absolutely continuous with respect to $\mu$, 
then $\sigma$ decomposes as a sum of countably many ergodic components, which are all of 
the form described in Theorem~\ref{thm:msj}.
\end{corollary}

To prove Theorem~\ref{thm:msj} in  the case $d=1$, we even do not need assumption~\eqref{eq:weaker} as we can show that 
$\mu$ is, up to a multiplicative constant, the only $T$-invariant, Radon measure 
on $X$ (the proof is the same as for the Chacon infinite transformation, see 
Proposition~2.4 in~\cite{ChaconInfinite}).

We also note that, if we have proved the second part of the 
theorem for some $d\ge2$, then the first one follows immediately. Indeed, if $\mu^{\otimes d}$ 
were not ergodic, then almost all its ergodic components would 
satisfy~\eqref{eq:weaker}, hence would be a product of graph measures different 
from $\mu^{\otimes d}$. But this would mean that for $\mu^{\otimes d}$-almost 
all $x\in X^d$, there exist at least two coordinates of $x$ lying on the same 
$T$-orbit, which of course is absurd. Hence $\mu^{\otimes d}$ is ergodic, and by 
Proposition~\ref{prop:dissipative}, it is conservative.

\bigskip

The remainder of this paper is devoted to the proof by induction of the second part of 
Theorem~\ref{thm:msj}. So we now assume 
that for some $d\ge2$, the statement is true up to $d-1$. We consider a nonzero, 
Radon, $\td$-invariant and ergodic measure  $\sigma$ on $X^d$, 
satisfying~\eqref{eq:weaker}.

We will show that either $\sigma$ is  a graph measure arising from powers of $T$, or it can be 
decomposed into a product of two measures $\sigma_1\times\sigma_2$, $\sigma_i$ being 
a $T^{\times d_i}$-invariant Radon measure on $X^{d_i}$ for some $1\le d_i<d$, $d_1+d_2=d$.
In this latter case we can apply the induction hypothesis to each $\sigma_i$, which yields the announced
result.

\subsection{Choice of a $\sigma$-typical point}
\label{sec:typicalpoint}
By Proposition~\ref{prop:dissipative}, the system $(X^d,\sigma,\td)$ is 
conservative ergodic.
 By Hopf's ergodic theorem, if $B\subset C\subset X^d$ with 
$0<\sigma(C)<\infty$, we have for $\sigma$-almost every point 
$x=(x_1,\ldots,x_d)\in X^d$
\begin{equation}
  \label{eq:Hopf}
  \dfrac{\sum_{j\in J}\ind{B}((\td)^jx)}{\sum_{j\in J}\ind{C}((\td)^jx)}
  \tend{|J|}{\infty} \dfrac{\sigma(B)}{\sigma(C)},
\end{equation}
where the sums in the above expression range over a set $J$ of consecutive 
integers containing 0.

 We say that $x\in X^d$ is \emph{typical} if, for all $n$ large enough so that 
$\sigma(\cd_n)>0$, Property~\eqref{eq:Hopf} holds whenever $B$ is an $n$-box and 
$C$ is $\cd_n$.
We know that $\sigma$-almost every $x\in X^d$ is typical. Therefore, there exists a point
$x=(x_1,\ldots,x_d)$ such that
\begin{equation}
  \label{eq:typical1}
  \begin{minipage}[t]{0.9\textwidth}
    For each $j\in\ZZ$, $(\td)^jx$ is typical.
  \end{minipage}
\end{equation}
Since there are only countably many boxes, we may also assume that
\begin{equation}
  \label{eq:typical2}
  \begin{minipage}[t]{0.9\textwidth}
    For each box $B$,  $x\in B\Longrightarrow\sigma(B)>0$.
  \end{minipage}
\end{equation}
Moreover, by~\eqref{eq:weaker}, we can further assume that 
\begin{equation}
  \label{eq:typical3}
  \begin{minipage}[t]{0.9\textwidth}
    $\forall i=1,\ldots,d$, $x_i\in X_\infty$.
  \end{minipage}
\end{equation}
From now on, we fix a point $x=(x_1,\ldots,x_d)$ satisfying the above assumptions 
\eqref{eq:typical1}, \eqref{eq:typical2} and \eqref{eq:typical3}. We will derive properties 
of $\sigma$ from the observations made on the orbit of this point $x$.

By an \emph{interval}, we mean in this paper a finite set of consecutive integers.
We will need the following key notion in our argument. 
\begin{definition}
\label{def:ncrossing}
  We call \emph{$n$-crossing} a maximal interval $J\subset \ZZ$ with the following properties:
  \begin{itemize}
    \item $(\td)^j x\in \cd_n$ for each $j\in J$,
    \item for each $1\le i\le d$, $j\mapsto t_n(T^jx_i)$ is constant on $J$.
  \end{itemize}
  An $n$-crossing is said to be \emph{synchronized} if $t_n(T^jx_1)=\cdots=t_n(T^jx_d)$ 
for each $j$ in this $n$-crossing.
\end{definition}

  Note that an $n$-crossing has at most $h_n$ elements. If $j$ is 
the smallest (respectively the largest) element of an $n$-crossing, then there 
exists $1\le i\le d$ such that $T^{j}x_i$ is in the first  (respectively the 
last) level of tower $n$. Observe also that when $j$ runs over an $n$-crossing, 
$(\td)^j x$ successively passes through each $n$-box of some $n$-diagonal.

\subsection{Characterizations of graph measures arising from powers of $T$}
\label{sec:graph}

\begin{lemma}
\label{lemma:graph}
  The following assertions are equivalent:
  \begin{enumerate}[(i)]
    \item $\sigma$ is a graph measure arising from powers of $T$;
    \item $\exists e_2,\ldots,e_d\in\ZZ$: $x_i=T^{e_i}x_1$ for each $i=2,\ldots,d$;
    \item $\exists\nb:\forall n\ge\nb$, $t_n(x_1)=\cdots=t_n(x_d)$;
    \item $\exists j,\exists\nb:\forall n\ge\nb$, $t_n(T^jx_1)=\cdots=t_n(T^jx_d)$.
  \end{enumerate}
\end{lemma}

\begin{proof}
  Let us first prove that (i) $\Longrightarrow$ (ii). If $\sigma$ is a graph measure arising from
  powers of $T$, then there exist a positive real number $\alpha$ and integers $e_2,\ldots,e_d$ such that 
  for all measurable subsets $A_1,\ldots,A_d$ of $X$, \eqref{eq:graph_measure} holds. 
  Observe that, if $\ell$ is large enough so that $h_{n_\ell - k(\ell)} > \max\{|e_2|,\ldots,|e_d|\}$, then
  for each $i=2,\ldots,d$ and each $j,j'\in\{0,\ldots,h_{n_\ell}-1\}$, 
  \[
    L_{n_\ell}^j \cap T^{-e_i}L_{n_\ell}^{j'}=\begin{cases}
                                           L_{n_\ell}^j &\text{ if }j'=j+e_i,\\
                                           \emptyset &\text{ otherwise}.
                                         \end{cases}
  \]
It follows that the only $n_\ell$-boxes that may be charged by $\sigma$ are of the form
$L_{n_\ell}^{j_1}\times L_{n_\ell}^{j_1+e_2}\times \cdots \times  L_{n_\ell}^{j_1+e_d}$ for some $j_1$. 
By assumption~\eqref{eq:typical2}, it follows that for each $i=2,\ldots,d$, $j_{n_\ell}(x_i)= j_{n_\ell}(x_1)+e_i$.
Since this is true for all large enough $\ell$, this in turn implies that for each $i=2,\ldots,d$, $x_i=T^{e_i}x_1$. 

Conversely, if (ii) holds, the same argument shows that if  $\ell$ is large enough so that 
$h_{n_\ell - k(\ell)} > \max\{|e_2|,\ldots,|e_d|\}$, then the only $n_\ell$-boxes that can contain $x$ are
the $n_\ell$-boxes of the form $L_{n_\ell}^{j_1}\times L_{n_\ell}^{j_1+e_2}\times \cdots \times  L_{n_\ell}^{j_1+e_d}$ for some $j_1$. 
Note that the $n_\ell$-boxes of this form constitute an $n_\ell$-diagonal, which we denote by $D$.
But (ii) is also valid for each $(\td)^jx$, $j\in\ZZ$ hence the argument also applies to each $(\td)^jx$. Thus, if $B$ is an 
$n_\ell$-box which is not on $D$, then $(\td)^jx\notin B$ for each $j\in\ZZ$. 
Now, remembering that $x$ is typical for $\sigma$, we have for each $n_\ell$-box $B=L_{n_\ell}^{j_1}\times\cdots\times L_{n_\ell}^{j_d}$
 \[
    \dfrac{\sigma(B)}{\sigma(\cd_{n_\ell})} = \lim_{k\to\infty} \dfrac{\sum_{-k\le j\le k}\ind{B}((\td)^jx)}{\sum_{-k\le j\le k}\ind{\cd_{n_\ell}}((\td)^jx)}.
 \]
 The above limit is 0 if $B$ is not on $D$. Moreover, note that each time the orbit of $x$ passes through $\cd_{n_\ell}$, $(\td)^jx$ successively passes
 through each $n_\ell$-box on $D$. Hence if $B$ is on $D$, the limit is equal to the inverse of the number of $n_\ell$ boxes 
 on $D$. In particular the limit is proportional to 
 $\mu\left(L_{n_\ell}^{j_1}\cap T^{-e_2}L_{n_\ell}^{j_1}\cap\cdots\cap T^{-e_d}L_{n_\ell}^{j_d}\right)$. The coefficient of proportionality depends
 \textit{a priori} on $\ell$, but since each $n_\ell$-box is a union of disjoint $n_{\ell+1}$-boxes, we see that in fact this coefficient does not depend on $\ell$.
 Finally, this gives~\eqref{eq:graph_measure} in the case of an $n_\ell$-box for each large enough $\ell$, and this is enough to conclude that \eqref{eq:graph_measure}
 holds for each measurable set of the form $A_1\times\cdots\times A_d$. We have so far proved the equivalence of (i) and~(ii).

Now let us turn to the proof of (ii) $\Longrightarrow$ (iii). Since $x_1\in X_\infty$, we have $j_n(x_1)\to \infty$ and 
$h_n-j_n(x_1)\to \infty$ as $n\to\infty$. If (ii) holds, we then have $j_n(x_i)=j_n(x_1)+e_i$ for each $i=1,\dots,d$ and each
$n$ large enough so that $\min\{j_n(x_1),h_n-j_n(x_1)\}>\max\{|e_2|,\ldots,|e_d|\}$. But then for such an $n$ we also have 
$t_n(x_i)=t_n(x_1)$ for each $i=1,\ldots,d$.

The implication (iii) $\Longrightarrow$ (iv) is obvious.

Assume now that (iv) holds with $j=0$ (\textit{i.e.} that, in fact, (iii) holds). For $i=2,\ldots,d$, we then have by an easy
induction that $j_n(x_i)-j_n(x_1)=j_{\nb}(x_i)-j_{\nb}(x_1)$ for each $n\ge \nb$. Setting $e_i\egdef j_{\nb}(x_i)-j_{\nb}(x_1)$
for $i=2,\ldots,d$, we get that $x_i=T^{e_i}x_1$ and we have (ii). Now if (iv) holds with some $j\in\ZZ$, we get (ii) for $(\td)^jx$, 
which is clearly equivalent to (ii) for $x$. Thus we have proved that (iv) $\Longrightarrow$ (ii) and this concludes the proof of the lemma.
\end{proof}

For the remainder of the paper, we also fix a real number $0<\eta<1$, 
small enough so that  $\eta < \frac{1}{100 d}$.  In particular we will need the inequality $(1-\eta)^2>1/2$.

\newcommand{\ci}{\boldsymbol I}
\begin{definition}
For each $n$, let $\ci_n\egdef \{-\lfloor h_n/2\rfloor,\ldots,-\lfloor h_n/2\rfloor+h_n-1\}$ be the interval of length $h_n$
and centered at 0.
  For each $n\ge0$, we call \emph{substantial $n$-crossing} any $n$-crossing whose intersection with  $\ci_n$ 
  countains at least $\eta h_n$ elements.
\end{definition}

\begin{lemma}
  \label{lemma:substantialncrossings}
  If $n=n_{(\ell-1)}$ for some large enough $\ell$, then substantial $n$-crossings 
  cover a proportion at least $(1-(d+2)\eta)$ of $\ci_n$. In particular, there exists at least one substantial $n$-crossing.
  Moreover, if all substantial $n$-crossings 
  are synchronized, then  each substantial $n$-crossing is of size at least $(1-(d+2)\eta) h_n$, and 
  there are at most two of them.
\end{lemma}

\begin{proof}
Let us start by considering the case of an integer $n$ which is of the form $n=n_{(\ell-1)}$ for some 
$\ell>\max_i \ell(x_i)$.   We also assume that $\ell$ is large enough so that 
\[
  \frac{1}{3^{k(\ell-1)}} < \frac{\eta}{2d}.
\]
We set $n'\egdef n_{(\ell-1)}-k(\ell-1)$, and we observe that the above assumption ensures that 
\begin{equation}
  \label{eq:n'n}
  \frac{h_{n'}+1}{h_n}< \frac{\eta}{d}.
\end{equation}
We know by Lemma~\ref{lemma:Xinfty} that $x\in \cd_{n_{(\ell-1)}}=\cd_n$,
and that the interval 
\[ \{-100 h_{(n_\ell-\ell)},\ldots ,100 h_{(n_\ell-\ell)}\}\]
is contained in a single $n_\ell$-crossing.
\textit{A fortiori}, $\ci_n$ is contained  in a single $n_\ell$-crossing.
Therefore, if a coordinate $T^jx_i$ reaches the top of tower~$n$ and comes back to $C_n$ on the interval $\ci_n$,
then the two passages in $C_n$ are separated by at most $h_{n'}+1$. 
Moreover, this can happen at most once on the interval $\ci_n$ for each $i$.
It follows that the set of integers $j\in\ci_n$ such that 
$(\td)^j x \notin \cd_n$ is constituted of at most $d$ pieces, and its cardinality is bounded above by $\eta h_n$ by~\eqref{eq:n'n}.
Then there exist at most $(d+1)$ $n$-crossings intersecting $\ci_n$, and they cover a proportion at least $(1-\eta)$ of $\ci_n$.
Now the proportion of $\ci_n$ covered by $n$-crossings which are not substantial is less than $(d+1)\eta$,
hence the proportion of $\ci_n$ covered by substantial $n$-crossings is at least $(1-(d+2)\eta)$. This proves the first part of the lemma

Let us assume now that all substantial $n$-crossings are synchronized.
If we have only one substantial $n$-crossing, then this $n$-crossing is of size at least $(1-(d+2)\eta) h_n$, 
and we have for $j$ in this $n$-crossing
\begin{equation}
  \label{eq:synchronized}
  \left|j_n(T^j x_{i_1})-j_n(T^j x_{i_2})\right| \le (d+2)\eta h_n.
\end{equation}
If we have at least two substantial $n$-crossings, note that between two of them, there is at least one coordinate passing through the 
top of tower~$n$, and for which $t_n$ has increased by 1 $\bmod\ 3$. 
Since the $t_n(T^j x_i)$, $i=1,\ldots,d$ are supposed to be equal on each substantial 
$n$-crossing, we deduce that \emph{each} coordinate passes through the top of tower $n$ between two substantial 
$n$-crossings. As this happens at most once for each coordinate on $\ci_n$, we see that there are at most two substantial
$n$-crossings. Finally, from the first part of the lemma it follows that two consecutive substantial $n$-crossings are 
separated by at most $(d+2)\eta h_n$ points. We deduce  that, on any 
substantial $n$-crossing, \eqref{eq:synchronized} holds, hence each substantial $n$-crossing
is of size at least $(1-(d+2)\eta) h_n$.
\end{proof}

\begin{remark}
  The preceding lemma extends easily to the case when $n_{(\ell-1)}\le n\le n_\ell-\ell$. Indeed, when  $n_{(\ell-1)}+1\le n\le n_\ell-\ell$
  the proof is even simpler, as two successive passages in $C_n$ are now separated by at most one. 
\end{remark}

\begin{prop}
  \label{prop:synchronized}
  The measure $\sigma$ is a graph measure arising from powers of $T$ if and only if for each large enough $n$, all substantial $n$-crossings are synchronized.
\end{prop}

\begin{proof}
  First assume that $\sigma$ is a graph measure arising from powers of $T$. 
  Then by Lemma~\ref{lemma:graph}, we know that there exists $e_2,\ldots,e_d\in\ZZ$ such that $x_i=T^{e_i}x_1$ for each $i=2,\ldots,d$.
  Take $n$ large enough so that $\max\{|e_2|,\ldots,|e_d|\}<\eta h_n$. Let $J$ be a substantial $n$-crossing. In particular the size of $J$
  is at least $\eta h_n$. Hence there exists $\jb\in J$ such that 
  $\eta h_n\le j_n(T^{\jb} x_1) \le (1-\eta)h_n$. We deduce that $j_n(T^{\jb} x_i)= j_n(T^{\jb} x_1)+e_i$ for each $i=2,\ldots,d$. But we also have 
  $\eta h_n\le j_{n+1}(T^{\jb} x_1) \le (1-\eta)h_n$ and this ensures that  $j_{n+1}(T^{\jb} x_i)= j_{n+1}(T^{\jb} x_1)+e_i$.
  By~\eqref{eq:induction_for_j_n},  the equality $j_{n+1}(T^{\jb} x_i) - j_{n+1}(T^{\jb} x_1) = j_{n}(T^{\jb} x_i) - j_{n}(T^{\jb} x_1)$
  implies $t_n(T^{\jb} x_i)=t_n(T^{\jb} x_1)$. Finally, as $j\mapsto t_n(T^jx_i)$
  is constant on the $n$-crossing $J$, we see that $J$ is synchronized. 
  
  Conversely, assume that there exists $\nb$ such that for each $n\ge\nb$, all substantial $n$-crossings are synchronized. Without loss of generality, we may 
  assume that $\nb$ is of the form $n_{(\ell-1)}$, for some $\ell$ large enough to apply Lemma~\ref{lemma:substantialncrossings}. Then we know that there
  exists at least one substantial $\nb$-crossing $J_{\nb}$, of size at least $(1-(d+2)\eta)h_{\nb}$. 
  For $j\in J_{\nb}$ and for each $i=2,\ldots,d$, $|j_{\nb}(T^j x_i)-j_{\nb}(T^j x_1)|\le (d+2)\eta h_{\nb}$.
  Let us prove by induction  that for each  $n\ge\nb$, there exists 
  a substantial $n$-crossing $J_n$, of size at least $(1-(d+2)\eta)h_n$, and containing $J_{\nb}$. We already know that this property is true for $\nb$. 
  Assume it is true up to $n$ for some $n\ge\nb$. Then, the $n$-crossing $J_n$ extends to a unique $(n+1)$-crossing $J_{n+1}$. 
  As $J_n$ intersects $\ci_n$ and is of size at most $h_n$, $J_n\subset \ci_{n+1}$. It follows that 
  \[
    \left|J_{n+1}\cap  \ci_{n+1}\right| \ge \left| J_n \right| \ge (1-(d+2)\eta)h_n \ge \eta h_{n+1},
  \]
  which proves that $J_{n+1}$ is a substantial $(n+1)$-crossing. 
  Moreover, since the size of $J_n$ is  at least $(1-(d+2)\eta)h_n$, we have for $j\in J_n$ and each $i=2,\ldots,d$ $|j_n(T^j x_i)-j_n(T^j x_1)|\le (d+2)\eta h_{n}$.
  But $J_n$ is synchronized, hence by~\eqref{eq:induction_for_j_n}, we have for $j\in J_n$
  \[
    |j_{n+1}(T^j x_i)-j_{n+1}(T^j x_1)| = |j_n(T^j x_i)-j_n(T^j x_1)| \le (d+2)\eta h_{n} \le (d+2)\eta h_{n+1}.
  \]
  This equality extends to $j\in J_{n+1}$ since the difference is constant on an $(n+1)$-crossing.
This proves that the size of $J_{n+1}$ is at least $(1-(d+2)\eta)h_{n+1}$. 
Now if we take any $j\in J_{\nb}$, we have $j\in J_n$ for each $n\ge\nb$, and since we assumed that each substantial $n$-crossing is synchronized, 
we have $t_n(T^jx_1)=\cdots=t_n(T^jx_d)$, \textit{i.e.} we have (iv) of Lemma~\ref{lemma:graph}. This proves that $\sigma$ is a graph measure arising from powers of $T$.
\end{proof}

\begin{remark}
\label{remark:sizes}
  In the preceding proof, the induction provides in fact a stronger inequality for the sizes of the substantial $n$-crossings $(J_n)$: 
  $|J_n|\ge h_n-(d+2)\eta h_{\nb}$.
\end{remark}

\section{Combinatorics of some sets of integers}
\label{sec:combinatorics}
The purpose of this section is to establish Proposition~\ref{prop:combinatorics_of_ncrossings} on the 
combinatorics of the set of integers $j$ such that $(\td)^j x\in \cd_n$ for a 
given large $n$.

%

\begin{prop}
  \label{prop:combinatorics_of_ncrossings}
  There exist constants $K_1>0$ and $K_2>0$ such that, for any large enough integer $\lb$, 
and any integer $1\le c\le h_{n_{\lb}}$, the following holds: 
  if $I\subset \ZZ$ is an interval contained in an $n_{(\lb+\ell)}$-crossing for 
some $\ell\ge1$, and if the length of $I$ is at least $\eta h_{n_{(\lb+\ell-1)}}$,
  then
  \begin{itemize}
    \item the proportion of integers $j\in I$ such that 
$(\td)^jx\in\cd_{n_{\lb}}$ is at least $(1-\eta)^{2\ell}$;
    \item among all the integers $j\in I$ such that $(\td)^jx\in\cd_{n_{\lb}}$, 
the proportion of those belonging to an $n_{\lb}$-crossing of size $\le c$ 
is 
   bounded above by $$K_1 \frac{c}{h_{n_{\lb}}} + \frac{K_2}{3^{\lb}}.$$
  \end{itemize}
\end{prop}

For this we will introduce a hierarchy of more and more complex subsets of 
$\ZZ$,  prove by induction some combinatorial results on abstract sets in this 
hierarchy,
and finally show how to apply these results in the particular case we are 
interested in.

\subsection{A hierarchy of subsets of $\ZZ$}
\label{sec:hierarchy}
This part of the argument is completely abstract and independent of the rest of 
the paper, but we keep the notations $d$ (an integer, $d\ge2$) and $\eta$
(a positive real number between 0 and 1).
We set 
\[
  K_1\egdef \dfrac{1+\frac{2}{\eta}}{1-\eta}d.
\]

We  fix two  sequences of positive integers $\left(c_\ell\right)_{\ell\ge1}$ and 
$\left(s_\ell\right)_{\ell\ge1}$, satisfying
\begin{equation}
  \label{eq:cs_cond1}
  \forall \ell\ge1,\ \frac{s_\ell}{c_\ell}<\frac{1}{d}\frac{\eta}{\eta+1}\eta,
\end{equation}
and 
\begin{equation}
  \label{eq:cs_cond2}
  \forall \ell\ge1,\ \frac{c_\ell}{c_{\ell+1}}<\frac{\eta}{K_1}.
\end{equation}

 Let $F\subset\ZZ$, and let $I\subset\ZZ$ be an interval. We call
\emph{piece} of $F\cap I$ any maximal interval included in $F\cap I$, and we 
call \emph{hole} of $F\cap I$ any maximal interval included in $I\setminus F$.
(Thus, $I$ is the disjoint union of the pieces and the holes of $F\cap I$, which 
alternate.)

We say that \emph{$F$ is of order 1 inside the interval $I$} if 
\begin{itemize}
  \item each hole of $F\cap I$ is of size $\le s_1$,
  \item two consecutive holes of $F\cap I$ are always separated by a piece of 
size at least $c_1$.
\end{itemize}

\begin{figure}
  \includegraphics[width=8cm]{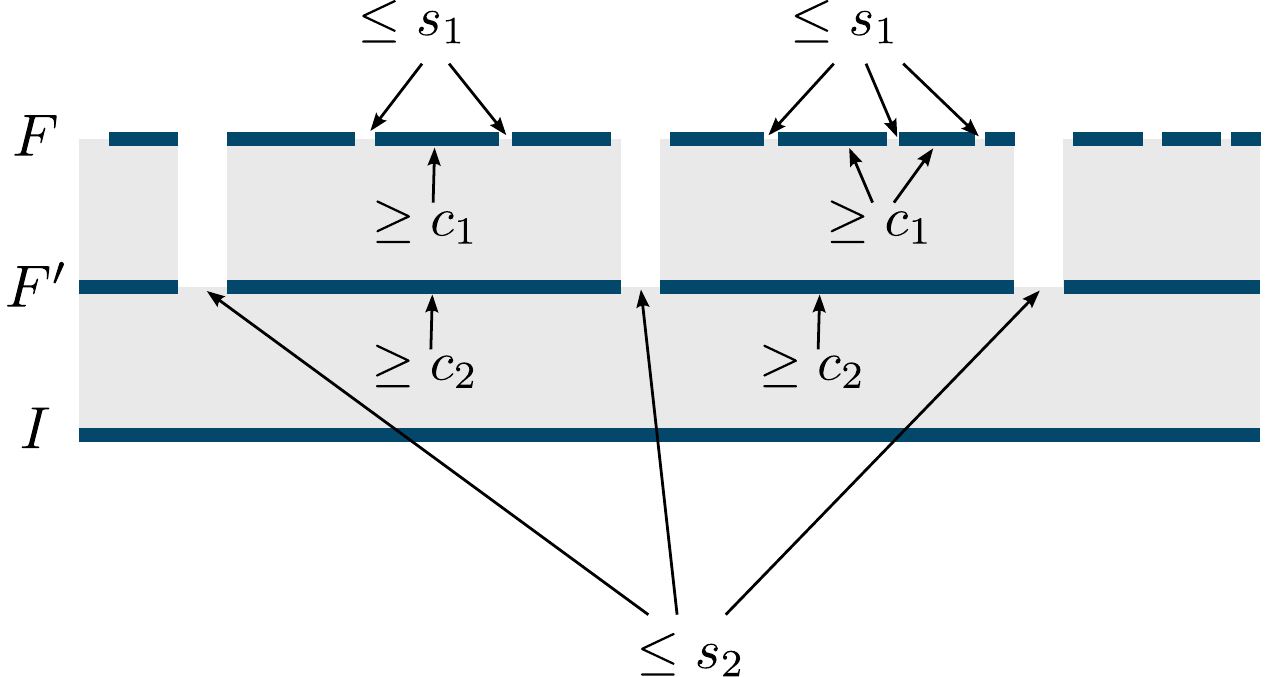}
\caption{A set $F$ of order 2 inside an interval $I$.}
\label{fig:order}
\end{figure} 

Recursively, we say that \emph{$F$ is of order $\ell\ge2$ inside the interval 
$I$} if there exists a subset $F'\subset \ZZ$ such that
\begin{itemize}
  \item $F \subset F'$,  
  \item each hole of $F'\cap I$ is of size $\le s_\ell$,
  \item two consecutive holes of $F'\cap I$ are always separated by a piece of 
size at least $c_\ell$,
  \item for each piece $I'$ of $F'\cap I$, $F$ is of order $(\ell-1)$ 
inside~$I'$.
\end{itemize}
(See Figure~\ref{fig:order}.)
Note that, if $F$ is of order $\ell$ inside the interval $I$, then $F$ is of 
order $\ell$ inside each subinterval $J\subset I$.

\begin{lemma}
  \label{lemma:combinatorics}
  Let $F_1,\ldots,F_d$ be $d$ subsets of $\ZZ$, and let $I\subset\ZZ$ be an 
interval. Assume that for some $\ell\ge1$, $F_i$ is of order $\ell$ inside $I$ 
for each $i=1,\ldots,d$, 
  and that the size of $I$ is at least $\eta c_\ell$. Set $F\egdef\bigcap_{i=1}^d 
F_i$.
  Then
  \begin{itemize}
    \item the density of $F$ inside $I$ satisfies
    \begin{equation}
      \label{eq:density_of_F}
      \frac{|F\cap I|}{|I|} \ge (1-\eta)^{2\ell},
    \end{equation}
  \item for a given integer $c$, $1\le c<h_1$, the proportion of integers in 
$F\cap I$ lying in pieces of $F\cap I$ with size $\le c$
  is bounded above by 
  \[
    K_1\left( \frac{c}{c_1} +   \frac{c_1}{c_2} \frac{1}{(1-\eta)^4} + \cdots + 
\frac{c_{\ell-1}}{c_\ell} \frac{1}{(1-\eta)^{2\ell}}\right)
  \]
  \end{itemize}

\end{lemma}

\begin{proof}
  Let us first establish the result for $\ell=1$. We assume that each $F_i$ is 
of order~1 inside $I$, and that $|I|\ge \eta c_1$. 
For each $i=1,\ldots,d$, let $k_i$ be the number of holes of $F_i\cap I$. Then 
 by definition of order~1, there are at least $k_i-1$ pieces of $F_i\cap I$ 
with size at least $c_1$, whence $c_1(k_i-1)\le |I|$, and
\[ k_i\le |I|/c_1+1.\]
  Since each hole of $F_i\cap I$ has size $\le s_1$, we deduce that the 
cardinality of $I\setminus F_i$ is bounded by $s_1(|I|/c_1+1)$.
  This yields by the inequality $1\le \frac{1}{\eta}\frac{|I|}{c_1}$ and~\eqref{eq:cs_cond1}:
  \[ |I\setminus F|   \le ds_1(|I|/c_1+1) \le d\left(1+\frac{1}{\eta}\right)s_1|I|/c_1 \le \eta |I|. \]
  We thus get $|F\cap I|/|I| \ge 1-\eta\ge (1-\eta)^{2}$, which is the first 
point. Moreover, the number $k$ of holes of $F\cap I$ satisfies
  $k\le k_1+\cdots+k_d\le d|I|/c_1+d$, whence the number $m$ of pieces of $F\cap 
I$ satisfies 
  \[ m \le d|I|/c_1+d+1  \le d|I|/c_1+2d \le d\left(1+\frac{2}{\eta}\right)|I|/c_1. \] 
  It follows that the number $r$ of points of $F\cap I$ lying 
  in a piece of size $\le c$ satisfies
  \[ r\le mc \le d\left(1+\frac{2}{\eta}\right)|I| \frac{c}{c_1}.\]
  As we already know that $|F\cap I|\ge (1-\eta)|I|$, we get by definition of $K_1$
  \[
    \frac{r}{|F\cap I|} \le \frac{\left(1+\frac{2}{\eta}\right)}{1-\eta} d \frac{c}{c_1} = K_1  
\frac{c}{c_1},
  \]
  which establishes the second point for $\ell=1$.
  
  Now we assume by induction that the result is true up to $\ell-1$ for some 
$\ell\ge2$ and we consider a family $(F_i)_{1\le i\le d}$ of subsets of $\ZZ$, 
  which are of order $\ell$ inside an interval $I$ satisfying $|I|\ge \eta c_\ell$. 
By definition of order $\ell$, for each $i$ there exists a subset 
$F'_i\subset\ZZ$ satisfying
  \begin{itemize}
  \item $F_i\subset F'_i$,
  \item each hole of $F'_i\cap I$ is of size $\le s_\ell$,
  \item two consecutive holes of $F'_i\cap I$ are always separated by a piece of 
size at least $c_\ell$,
  \item for each piece $I'$ of $F'_i\cap I$, $F'_i$ is of order $(\ell-1)$ 
inside~$I'$.
\end{itemize}
Since $|I|\ge \eta c_\ell$, the argument developped for order~1 applies for $F'\egdef \bigcap_{i=1}^d F'_i$ 
(with $(c_{\ell-1},c_\ell,s_\ell)$ in place of $(c,c_1,s_1)$). We thus get
 \begin{equation}
   \label{eq:size_of_F'}
   |F'\cap I| \ge (1-\eta)|I|,
 \end{equation}
 and denoting by $r'$ the number of points of $F'\cap I$ lying in pieces of 
$F'\cap I$ of size  $<c_{\ell-1}$, we have (using 
also~\eqref{eq:cs_cond2})
  \begin{equation}
    \label{eq:r'} \frac{r'}{|F'\cap I|} \le  K_1  \frac{c_{\ell-1}}{c_\ell}<\eta.
  \end{equation}
Let $G$ stand for the union of all pieces of $F'\cap I$ of size $\ge c_{\ell-1}$. 
The above inequality can be rewritten as
\begin{equation}
  \label{eq:size_of_G}
  \frac{|G|}{|F'\cap I|} > (1-\eta).
\end{equation}
Let $J$ be an arbitrary piece of $G$.
Since for each $i$, $F_i$ is of order $\ell-1$ inside $J$, and by definition of 
$G$, $|J|\ge c_{\ell-1}\ge \eta c_{\ell-1}$,
the induction hypothesis gives  
\[\frac{|F\cap J|}{|J|} \ge (1-\eta)^{2{\ell-2}}.\]
Summing over all pieces of $G$ we get, using also~\eqref{eq:size_of_G} 
and~\eqref{eq:size_of_F'}
\begin{equation}
  \label{eq:density}
  |F\cap I|\ge |F\cap G| \ge (1-\eta)^{2{\ell-2}} |G| \ge (1-\eta)^{2\ell}|I|,
\end{equation}
which is the first point at order $\ell$.

Moreover, if $r_J$ denotes the number of points of $F\cap J$ lying in pieces of 
$F\cap J$ of size smaller than $c$, then
  \[
    \frac{r_J}{|F\cap J|} \le   K_1\left( \frac{c}{c_1} +   \frac{c_1}{c_2} 
\frac{1}{(1-\eta)^4} + \cdots + \frac{c_{\ell-2}}{c_{\ell-1}} 
\frac{1}{(1-\eta)^{2{\ell-2}}}\right).
  \]
Now let us denote by $r$ the number of points of $F\cap I$ lying in pieces of 
$F\cap I$ of size smaller than $c$. The contribution to $r$ of points in $G$
is $\sum_J r_J$ (where the sum ranges over all pieces $J$ of $G$), and by the 
previous inequality, it satisfies
\begin{align*}
  \sum_J r_J & \le K_1\left( \frac{c}{c_1} +   \frac{c_1}{c_2} \frac{1}{(1-\eta)^4} + \cdots + \frac{c_{\ell-2}}{c_{\ell-1}} 
\frac{1}{(1-\eta)^{2{\ell-2}}}\right) |F\cap G| \\
& \le K_1\left( \frac{c}{c_1} +   \frac{c_1}{c_2} 
\frac{1}{(1-\eta)^4} + \cdots + \frac{c_{\ell-2}}{c_{\ell-1}} 
\frac{1}{(1-\eta)^{2{\ell-2}}}\right) |F\cap I|.
\end{align*}
 The contribution to $r$ of points in $F\setminus G$ is clearly at most $|F\setminus G|$, which can be bounded above
 as follows
 \begin{align*}
   |(F\cap I)\setminus G| & \le |(F'\cap I)\setminus G| \qquad\text{(because $F\subset F'$)}\\
   & = r'  \qquad\text{(by definition of  $G$ and $r'$)}\\
   & \le    K_1  \frac{c_{\ell-1}}{c_\ell}|F'\cap I| \qquad\text{(by \eqref{eq:r'})}\\
   & \le    K_1  \frac{c_{\ell-1}}{c_\ell}|I| \\
   & \le    K_1  \frac{c_{\ell-1}}{c_\ell} \frac{|F\cap I|}{(1-\eta)^{2\ell}}  \qquad\text{(by \eqref{eq:density})}
 \end{align*}
Summing the two contributions and using the above inequalities, we get 
\[
  r \le K_1\left( \frac{c}{c_1} +   \frac{c_1}{c_2} 
\frac{1}{(1-\eta)^4} + \cdots + \frac{c_{\ell-2}}{c_{\ell-1}} 
\frac{1}{(1-\eta)^{2{\ell-2}}}  + \frac{c_{\ell-1}}{c_\ell} \frac{1}{(1-\eta)^{2\ell}}  \right) |F\cap I|,
\]
which is the second point at order $\ell$.
\end{proof}

\subsection{Application to the structure of $n$-crossings}
\label{sec:application}
We want now to apply the preceding lemma in order to obtain some statistical 
results on long range of successive $n$-crossings.
We fix some integer $\lb$, large enough to satisfy some conditions to be 
precised later, and we set $\kb\egdef k(\lb)$
We define the sequences 
${(c_\ell)}_{\ell\ge1}$ and ${(s_\ell)}_{\ell\ge1}$ as follows.
\begin{itemize}
  \item $c_1\egdef h_{n_{\lb}}$,
  \item $s_1\egdef h_{(n_{\lb}-\kb)}+1$,
  \item in general, $c_\ell\egdef h_{n_{(\lb+\ell-1)}}$, and $s_\ell \egdef 
h_{n_{(\lb+\ell-1)}-k(\lb+\ell-1)} + 1$.
\end{itemize}

Using the fact that we always have $h_n/h_{n+1}<1/3$, we observe that for each 
$\ell\ge1$, with $k\egdef k(\lb+\ell-1)$,
\[
  \frac{s_\ell}{c_\ell} = \frac{h_{n_{(\lb+\ell-1)}-k} + 
1}{h_{n_{(\lb+\ell-1)}}} < 2 \frac{h_{n_{(\lb+\ell-1)}-k}}{h_{n_{(\lb+\ell-1)}}} 
< \frac{2}{3^k} \le \frac{2}{3^{\kb}}.
\]
Hence \eqref{eq:cs_cond1} is satisfied if $\lb$ is large enough. The fact that  
\eqref{eq:cs_cond2} holds if $\lb$ is large enough follows from the following 
easy consequence of~\eqref{eq:condition_nl}:
\[
  \frac{h_{n_\ell}}{h_{n_{\ell+1}}} < \frac{ h_{n_{(\ell+1)}}-(\ell+1) }{ 
h_{n_{(\ell+1)}} } < \frac{1}{3^{\ell+1}} \tend{\ell}{\infty} 0.
\]
We can therefore assume that $\lb$ is large enough so that both 
\eqref{eq:cs_cond1} and \eqref{eq:cs_cond2} hold.

We want to apply Lemma~\ref{lemma:combinatorics} to the subsets $F_i$ 
($i=1,\ldots,d$) defined by
\[
  F_i \egdef \left\{j\in\ZZ: T^j x_i \in C_{n_{\lb}} \right\}.
\]

Let $I\subset\ZZ$ be an interval, $n\ge1$ and $i\in\{1,\ldots,d\}$. We say that 
\emph{$x_i$ climbs into tower $n$ along $I$} if
for each $j\in I$, $T^jx_i\in C_n$, and there is no $j\in I$ such that $j+1\in 
I$, $T^jx_i\in L_n^{h_n-1}$ and $T^{j+1}x_i\in L_n^{0}$.
Note that $I$ is included in an $n$-crossing if and only if each coordinate 
$x_i$ climbs into tower $n$ along $I$.

\begin{lemma}
\label{lemma:order}
  For each interval $I\subset\ZZ$ and each $i\in\{1,\ldots,d\}$, if $x_i$ climbs 
into tower~$n_{(\lb+\ell)}$ along $I$, then $F_i$ is of order~$\ell$ inside $I$.
\end{lemma}

\begin{proof}
By construction of the Nearly Finite Chacon Transformation, two successive 
occurrences of tower~$n_{\lb}$ inside tower~$n_{(\lb+1)}$ are separated
either by $h_{n_{\lb}-\kb}$ or by $h_{n_{\lb}-\kb}+1$ spacers. Hence, if $x_i$ 
climbs into tower $n_{(\lb+1)}$ along $I$,
$F_i$ is of order~1 inside $I$. This proves the lemma in the case $\ell=1$.

Assume that the statement of the lemma is true up to $\ell-1$ for some 
$\ell\ge2$. We consider 
\[
  F'_i \egdef \left\{j\in\ZZ: T^j x_i \in C_{n_{(\lb+\ell-1)}} \right\}.
\]
We clearly have $F_i\subset F'_i$.

Two successive occurrences of tower~$n_{(\lb+\ell-1)}$ inside 
tower~$n_{(\lb+\ell)}$ are separated
either by $h_{n_{(\lb+\ell-1)}-k}$ or by $h_{n_{(\lb+\ell-1)}-k}+1$ spacers, 
where $k$ is determined by $\ell_k\le \lb+\ell-1 <\ell_{k+1}$. Hence, if $x_i$ 
climbs into tower $n_{(\lb+\ell)}$ along $I$,
each hole of $F'_i\cap I$ is of size $\le h_{n_{(\lb+\ell-1)}-k}+1=s_\ell$, 
and two consecutive holes of $F'_i\cap I$ are separated by a piece of $F'_i\cap 
I$ of size 
$h_{n_{(\lb+\ell-1)}}=c_\ell$. Moreover, along each piece of $F'_i\cap I$, $x_i$ 
climbs into tower~$n_{(\lb+\ell-1)}$. Therefore the property for $\ell-1$ 
ensures
that $F_i$ is of order $\ell-1$ inside each piece of $F'_i\cap I$. It follows that 
$F_i$ is of order $\ell$ inside I, and the lemma is proved by induction.
\end{proof}

\begin{proof}[Proof of Proposition~\ref{prop:combinatorics_of_ncrossings}]
With the subsets $F_i$ defined as above, we see that 
\[ 
    F\egdef\bigcap_{1\le i\le d}F_i = \left\{j\in\ZZ: (\td)^j x \in 
\cd_{n_{\lb}} \right\}.
\]
Observe that the pieces of $F$ are precisely the $n_{\lb}$-crossings.

Assume that the interval $I\subset\ZZ$ is included in an 
$n_{(\lb+\ell)}$-crossing for some $\ell\ge1$ (remember that this is equivalent 
to: each coordinate $x_i$ climbs into tower $n_{(\lb+\ell)}$ along $I$).
Then, putting together Lemma~\ref{lemma:combinatorics} and 
Lemma~\ref{lemma:order}, and provided that the length of $I$ be at least 
$\eta c_\ell=\eta h_{n_{(\lb+\ell-1)}}$, we get:
\begin{itemize}
  \item the proportion of $j\in I$ such that $(\td)^j x \in \cd_{n_{\lb}}$ is at 
least $(1-\eta)^{2\ell}$,
  \item for each $1\le c\le h_{n_{\lb}}$, the proportion of $j\in F\cap I$ 
belonging to an $n_{\lb}$-crossing of size $\le c$ is bounded above by
  \begin{equation}
  \label{eq:density_of_small_ncrossings}
     K_1 \left(
    \frac{c}{h_{n_{\lb}}} + \frac{ h_{n_{\lb}} }{ h_{n_{(\lb+1)}} } 
\frac{1}{(1-\eta)^4} + \cdots
    + \frac{ h_{n_{(\lb+\ell-2)}} }{ h_{n_{(\lb+\ell-1)}} } 
\frac{1}{(1-\eta)^{2\ell}}
    \right).
  \end{equation}
\end{itemize}
Let us estimate the general term of the above sum, using the inequality 
$h_{n_\ell}/h_{n_{(\ell+1)}}<1/3^{\ell+1}$, and the assumption
$(1-\eta)^2>1/2$.
\begin{align*}
   \frac{ h_{n_{(\lb+\ell-2)}} }{ h_{n_{(\lb+\ell-1)}} } 
\frac{1}{(1-\eta)^{2\ell}} 
            & < \frac{1}{3^{(\lb+\ell-1)}}  \frac{1}{(1-\eta)^{2\ell}}\\
            & =  \frac{1}{3^{(\lb -1)}}  \frac{1}{\left( 3 (1-\eta)^2 
\right)^\ell}\\
            & < \frac{1}{3^{(\lb -1)}} \left(\frac{2}{3}\right)^\ell.
\end{align*}
It follows that there exist a constant $K_2$ such that 
\eqref{eq:density_of_small_ncrossings}$\displaystyle \le K_1 
\frac{c}{h_{n_{\lb}}} + \frac{K_2}{3^{\lb}}$. 
\end{proof}

\subsection{Measure of the edge of $\cd_n$}
\label{sec:edge}
For each $n\ge0$, we say that an $n$-box $L_n^{j_1}\times\cdots\times L_n^{j_d}$ is \emph{on the edge of $\cd_n$} if 
there exists $i\in\{1,\ldots,d\}$ such that $j_i=0$ or $j_i=h_n-1$. We denote by $\partial \cd_n$ the union of all such 
$n$-boxes. 

As a first application of Proposition~\ref{prop:combinatorics_of_ncrossings}, we have the following result.

\begin{corollary}
  \label{cor:edge}
  \[
    \delta(n)\egdef \dfrac{\sigma(\partial \cd_n)}{\sigma(\cd_n)} \tend{n}{\infty} 0.
  \]
\end{corollary}

\begin{proof}[sketch of proof]
  This is a direct consequence of the following facts:
  \begin{itemize}
    \item Since $x$ is typical for $\sigma$, the quotient $\delta(n)$ can be estimated by the ratio 
    \[
      \dfrac{\sum_{j\in I}\ind{\partial \cd_n} (\td)^j x}{\sum_{j\in I}\ind{\cd_n} (\td)^j x}
    \]
  for a large interval $I$ containing 0.
  \item The subset of $j\in \ZZ$ such that $(\td)^jx\in \cd_n$ is partitioned into $n$-crossings, and in 
  each $n$-crossing $J$ there are exactly two integers $j$  (the minimum and the maximum of $J$) such that $(\td)^jx\in\partial\cd_n$.
  \item By Proposition~\ref{prop:combinatorics_of_ncrossings}, most $n$-crossings are large if $n$ is large.    
  \end{itemize}
\end{proof}

\section{Convergence of sequences of empirical measures}
\label{sec:convergence}
For each finite subset $J\subset\ZZ$, we denote by $\gamma_J$ the \emph{empirical 
measure}
\[
 \gamma_J \egdef \sum_{j\in I} \delta_{(\td)^j x}.
\]
The validity of Property~\eqref{eq:Hopf}  whenever $B$ is an $n$-box and $C$ is 
$\cd_n$ (remember that $x$ has been chosen as a typical point) means that, if
$(J_n)$ is a sequence of intervals containing $0$, with 
$|J_n|\tend{n}{\infty}\infty$, then we have
the convergence $\gamma_{J_n}\tend{n}{\infty} \sigma$.

Our purpose in this section is to extend this convergence to the case where the 
intervals $J_n$ do not necessarily contain 0, but are \emph{not too far} from 0.
We will also treat the case where the subsets $J_n$ are no longer intervals, but union of intervals
with a sufficiently regular structure.

We fix a real number $\varepsilon>0$, small enough so that 
$(1-\varepsilon)^2>1-\eta$. Then we consider an integer $c\ge1$, large enough so 
that 
$\frac{c-1}{c}>1-\varepsilon$.

In Sections~\ref{sec:4.1} and \ref{sec:4.2}, we consider a fixed integer $\underline{\ell}$, 
large enough so that the result of 
Proposition~\ref{prop:combinatorics_of_ncrossings} holds. We can also assume 
that
\begin{equation}
 \label{eq:choice_of_lb}
 K_1 \frac{c}{h_{n_{\lb}}} + \frac{K_2}{3^{\lb}} < \varepsilon.
\end{equation}
We are going to estimate the behaviour of empirical measures with respect to 
$n_{\lb}$-boxes. 
The following lemmas are devoted to the control of 
\[
  \Gam{I} = \sum_{j\in I} \ind{\cd_{n_{\lb}}} \left( (\td)^j(x)\right) 
\]
 for particular intervals $I$.

\subsection{Consecutive $n$-intervals}
\label{sec:4.1}
For $n\ge1$, we call \emph{$n$-interval} any interval 
$I=\{j,j+1,\ldots,j+h_n-1\}$ of length $h_n$ and such that $j$ is a multiple of $h_n$.
(The second condition is completely artificial, it is only useful to define canonically a cutting of any
interval into intervals of length $h_n$.)

\begin{lemma}
\label{lemma:consecutive_n_intervals} 
 Let $p_1$ be the smallest integer such that $3^{p_1}>2d+1$ and $p_1>d$. There exists a constant $0<\theta_1<1$
 (depending only on $\eta$ and $d$) for which the following holds.
 
 Let $\ell>\lb+1$, and let $n$ be such that  
 $n_{(\ell-1)}-k(\ell-1)+p_1\le n< n_\ell$.
 Whenever $I_1$ and $I_2$ are two consecutive 
 $n$-intervals, both contained in the same $n_\ell$-crossing, we have 
 \[ 
   \theta_1 \Gam{I_1} < \Gam{I_2} < \frac{1}{\theta_1} \Gam{I_1}. 
 \]
\end{lemma}

\begin{proof}
We divide the proof into two cases.

\subsubsection*{Case 1: $n_{(\ell-1)}+1\le n\le n_\ell$}

Set $j_1\egdef \min I_1$ and $j_2\egdef\min I_2=j_1+h_n$. 
 Proposition~\ref{prop:combinatorics_of_ncrossings} applies to $I_1$, and this 
ensures that, 
 among the $\Gam{I_1}$ integers $j$ such that $(\td)^{j_1+j}x\in 
\cd_{n_{\lb}}$, a proportion at least $(1-\varepsilon)$ 
(by~\eqref{eq:choice_of_lb}) belong to an 
 $n_{\lb}$-crossing of size at least $c$. Then, among those belonging to an 
$n_{\lb}$-crossing of size at least $c$, a proportion at least $\frac{c-1}{c}$ 
 are not the minimum of their $n_{\lb}$-crossing. By the 
choice of $\varepsilon$ and $c$, we get the partial following result: 
 a proportion at least $1-\eta$ of integers $j\in\{0,\ldots,h_n-1\}$ are such 
that, for each $i=1,\ldots,d$, $T^{j_1+j}x_i\in C_{n_{\lb}}$, but 
$T^{j_1+j}x_i$ 
 is not in the bottom level of tower~$n_{\lb}$. Let us 
consider such an integer $j$. Observe that, since $I_2$ is in the same $n_\ell$-crossing as 
$I_1$, the coordinate $T^{j_1+j}x_i$ cannot be in the last occurrence of tower $n$ inside tower 
$n_{\ell}$. Hence it will pass through zero or one spacer before coming back to $C_{n}$.
Then we can use a similar argument as in the proof of Lemma~\ref{lemma:dissipative_case}:
according to whether coordinate $i$ sees a spacer or not, $T^{j_2+j}x_i=T^{j_1+j+h_n}x_i$ is either 
in the same level of tower $n$ as $T^{j_1+j}x_i$, or in the level immediately below.
And the same applies if we consider the levels of tower~$n_{\lb}$. Hence $T^{j_2+j}x_i\in C_{n_{\lb}}$.

This proves that $\Gam{I_2}\ge (1-\eta) \Gam{I_1}$. But we can do a similar reasoning starting from $I_2$ 
and going backwards, and we get the announced inequalities for any $0<\theta_1\le (1-\eta)$ .

\subsubsection*{Case 2: $n_{(\ell-1)}-k(\ell-1)+p_1\le  n\le n_{(\ell-1)}$}
To simplify the notations, we set $n'\egdef n_{(\ell-1)}-k(\ell-1)$. The reason why we cannot do the same reasoning as
in the previous case is the following: when for some $j$ the coordinate $T^jx_i$ leaves tower~$n$, it will
come back to $C_n$ after 0, 1, $h_{n'}$ or $h_{n'}+1$ spacers. Because of this huge number of spacers that might separate
two climbings into tower~$n$, we cannot be sure that $T^{j+h_n}x_i$ will be in $C_{n_{\lb}}$. 
To circumvent this difficulty, we introduce what we call the \emph{fake tower $n'$}: it is the Rokhlin tower of height $h_{n'}$ 
whose levels are the $h_{n'}$ spacers placed on top of tower~$n_{(\ell-1)}$ in the construction. Let us denote by
$\tilde L_{n'}^0,\ldots,\tilde L_{n'}^{h_{n'}-1}$ its consecutive levels. We note that this fake tower $n'$ is disjoint 
from $C_{n_{(\ell-1)}}$, \textit{a fortiori} it is disjoint from $C_{n_{\lb}}$. 
However we can construct a \emph{fake $C_{n_{\lb}}$} inside the fake tower $n'$ by mimicking the structure of $C_{n_{\lb}}$ inside
tower~$n'$. More precisely, we set
\[
  \tilde C_{n_{\lb}} \egdef \bigsqcup_{j:L_{n'}^j\subset C_{n_{\lb}}} \tilde L_{n'}^j, \qquad\text{and}\qquad
  \overline{C}_{n_{\lb}}\egdef C_{n_{\lb}}\sqcup \tilde C_{n_{\lb}}.
\]
If we consider $\overline{C}_{n_{\lb}}$ instead of $C_{n_{\lb}}$, then everything happens as if the $n$-intervals 
were both contained in a single $n_{(\ell-1)}$-crossing. Hence we can use Case~1 with $(\ell-1)$ in place of $\ell$, which
yields 
\begin{equation}
  \label{eq:case1withfake}
   (1-\eta) \Gamb{I_1} < \Gamb{I_2} < \frac{1}{(1-\eta)} \Gamb{I_1}. 
\end{equation}
It remains now to compare $\Gamb{I}$ with $\Gam{I}$ for $I=I_1$ or $I=I_2$. 

For this we will consider $n'$-intervals intersecting $I$. Let $J$ be such an $n'$-interval. We say that it is \emph{suspect}
if there exists $1\le i \le d$ and $j\in J$ such that $T^jx_i \notin C_{n_{(\ell-1)}}$.
Note that, by definition of a suspect interval, if $J\subset I$ is an $n'$-interval which is not suspect, then $\Gam{J}=\Gamb{J}$.
Observe also that, when $T^jx_i$ leaves $C_{n_{(\ell-1)}}$, then it comes back after at most $h_{n'}+1$ spacers 
(remember that everything takes place inside an $n_{\ell}$-crossing, therefore the coordinates do not leave $C_{n_\ell}$). Moreover,
when it comes back to $C_{n_{(\ell-1)}}$, it stays in $C_{n_{(\ell-1)}}$ for a time $\ge h_{n_{(\ell-1)}}\ge h_n$.
Since $|I|=h_n$, each coordinate $1\le i\le d$ is responsible for at most 2 suspect $n'$-intervals intersecting $I$, and we 
conclude that there exist at most $2d$ suspect $n'$-intervals intersecting $I$. Moreover, since we assumed that $n\ge n'+p_1$, 
we have $h_n/h_{n'}>3^{p_1}>2d+1$, and this ensures that there exists at least one $n'$-interval contained in $I$ which is not suspect.

Now if $J'$ is a suspect interval intersecting $I$, we can find a chain $J'=J'_0,J'_1,\ldots,J'_{r}=J$ of consecutive $n'$-intervals,
where $J'_0,\ldots,J'_{r-1}$ are suspect (hence $r\le 2d$), $J=J'_r$ is not suspect and contained in $I$.
Applying Case~1 $r\le 2d$ times (with $n'$ in place of $n$ and $\ell-1$ in place of $\ell$ gives
\[
 \Gam{I} \ge \Gam{J} = \Gamb{J} \ge (1-\eta)^{2d} \Gamb{J'} \ge (1-\eta)^{2d}\Gamt{J'}.
\]
Since only suspect intervals can contribute to $\Gamt{I}$, and since there are at most $2d$ of them, summing the preceding 
inequality over all suspect intervals $J'$ intersecting $I$  yields
\[
 2d \Gam{I} \ge (1-\eta)^{2d}\Gamt{I}.
\]
In other words,
\[
  \Gamt{I} \le \frac{2d}{(1-\eta)^{2d}} \Gam{I}.
\]
Adding $\Gam{I}$ on both sides,  we get
\[
\Gamb{I} \le \left(   \frac{2d}{(1-\eta)^{2d}} + 1 \right) \Gam{I}.
\]
Inserting the above inequality for $I=I_1$ in~\eqref{eq:case1withfake}, we get
\[
  \Gam{I_2} \le \Gamb{I_2} < \frac{1}{\theta_1} \Gam{I_1}
\]
with 
\[
  \theta_1\egdef (1-\eta) \left(   \frac{2d}{(1-\eta)^{2d}} + 1 \right)^{-1}.
\]
But we can exchange the roles of $I_1$ and $I_2$ and this gives the annouced result.
\end{proof}

\begin{remark}
\label{rem:extremities}
  Let $\ell$ and $n$ be as in Lemma~\ref{lemma:consecutive_n_intervals}.  Assume that $I_1$ and $I_2$
  are consecutive $n$-intervals, but only $I_1$ is supposed to be contained in some $n_\ell$ crossing $J$. Then we
  get the inequality 
  \[
     \Gam{I_2\cap J} < \frac{1}{\theta_1} \Gam{I_1}. 
  \]
\end{remark}
Indeed, we can always change what happens on $I_2\setminus J$ to do as if both $I_1$ and $I_2$ were included in the same $n_\ell$-crossing.

\subsection{Contribution of substantial subintervals}

\label{sec:4.2}

\begin{lemma}
\label{lemma:ac}
 Let $p_2$ be the smallest integer such that $\frac{1}{3^{p_2}}<\frac{1}{3}\eta$. For each $M>0$, there exists a 
 real number $0<\theta_2(M)<1$
 (depending also on $\eta$ and $d$) for which the following holds.
 
Let $\ell>\lb+1$, and let $n$ be such that  
 \begin{equation}
 \label{eq:condition_n}
  n_{(\ell-1)} -k(\ell-1)+p_1+p_2 \le n  < n_\ell.
 \end{equation} 
 For each interval $I$ of length $|I|\le M h_n$  contained in an 
 $n_\ell$-crossing, for each subinterval $J\subset I$ with  $|J|\ge \eta h_n$, we have 
 \[\Gam{J}\ge \theta_2(M)\, \Gam{I}.\]
\end{lemma}

\begin{remark}
  Note that if $\ell$ is large enough, we have $n_{(\ell-1)} -k(\ell-1)+p_1+p_2<n_{(\ell-1)}$, hence the above is valid in particular 
  for $n_{(\ell-1)}\le n < n_\ell$.
\end{remark}

\begin{proof}
Under the assumptions of the lemma, we have 
\[
  n_{\lb}< n_{(\ell-1)} - k(\ell-1) + p_1 \le n-p_2 < n_\ell.
\]
We consider the $(n-p_2)$-intervals included in $I$, and we will apply Lemma~\ref{lemma:consecutive_n_intervals} to them.
 Remember that $h_{n-p_2}\ge \frac{1}{7^{p_2}}h_n$. Hence the number of $(n-p_2)$-intervals contained in $I$ is at most $7^{p_2} M$. Moreover their length $h_{n-p_2}$ satisfies $h_{n-p_2}<\frac{1}{3^{p_2}}h_n<\frac{1}{3}|J|$. 
 Hence there is at least one $(n-p_2)$-interval included in $J$. Let us call it $J'$. Now, if $I'$
 is another $(n-p_2)$-interval contained in $I$, a repeated use of Lemma~\ref{lemma:consecutive_n_intervals} yields 
 \[
   \Gam{J}\ge\Gam{J'} \ge \theta_1^{7^{p_2}M} \Gam{I'}.
 \]
There might also exist two $(n-p_2)$ intervals intersecting $I$ at its extremities but 
not contained in $I$, hence not necessarily contained in the $n_\ell$-crossing.
If $I'$ is such an interval, we use Remark~\ref{rem:extremities} and get the same inequality
(with $\gamma_{I'\cap I}$ instead of $\gamma_{I'}$).
Summing over all the $(n-p_2)$-intervals intersecting $I$, we get 
\[
 \left(7^{p_2}M +2\right)\Gam{J} \ge \theta_1^{7^{p_2}M} \Gam{I}.
\]
This gives the annouced result, with $\theta_2(M)\egdef \theta_1^{7^{p_2}M}/\left( 7^{p_2}M+2\right)$.
\end{proof}

\subsection{How to apply Proposition~\ref{prop:compacity}}

Here we want to provide some conditions so that Proposition~\ref{prop:compacity} applies to a sequence of 
empirical measures $\bigl(\gamma_{J_m}\bigr)$ for some sequence of intervals $(J_m)$. 
We make the following assumptions. 
\begin{equation}
  \label{eq:assumptionJm1}
  \begin{minipage}[t]{0.9\textwidth}
    For each $m$, there exists an integer $\ell_m$ with $\ell_m\to\infty$ as $m\to\infty$, such that 
$J_m$ is  contained in some $n_{\ell_m}$-crossing,
  \end{minipage}
\end{equation}
and  
\begin{equation}
  \label{eq:assumptionJm2}
  \begin{minipage}[t]{0.9\textwidth}
     For each $m$, there exists an integer $n(m)$ satisfying  
    \begin{itemize}
      \item  $n_{(\ell_m-1)}-k(\ell_m-1) + p_1+ 2p_2 \le n(m)< n_{\ell_m}$, 
      \item $\eta h_{n(m)}\le |J_m| \le h_{n(m)},$
    \end{itemize}
  \end{minipage}
\end{equation}

%

Note that, as soon as $\lb$ is large enough so that Proposition~\ref{prop:combinatorics_of_ncrossings} applies, the first point 
of this proposition ensures that $\gamma_{J_m}\bigl(\cd_{n_{\lb}}\bigr)>0$ for $m$ large enough, which is the first assumption
needed to apply Proposition~\ref{prop:compacity}.

It remains, for some fixed $\lb$ and $\ell$, to control the ratio 
$\gamma_{J_m}\bigl(\cd_{n_{(\lb+\ell)}}\bigr)/\gamma_{J_m}\bigl(\cd_{n_{\lb}}\bigr)$, which is the purpose of the following 
lemma.

\begin{lemma}
  \label{lemma:ratio}
  Let $\lb$ be large enough so that Proposition~\ref{prop:combinatorics_of_ncrossings} applies and \eqref{eq:choice_of_lb} holds. 
  Assume also that 
  \begin{equation}
    \label{eq:newcondlb}
    (K_1+K_2)\frac{1}{3^{\lb}}<\eta.
  \end{equation}
  Let $\ell\ge1$, and let $(J_m)$ be a sequence of  intervals satisfying \eqref{eq:assumptionJm1} and~\eqref{eq:assumptionJm2}.
  Then for each $m$ large enough
  \[
    \gamma_{J_m}\bigl(\cd_{n_{\lb}}\bigr) \ge (1-\eta)^{2\ell+1} \theta_2(7)\gamma_{J_m} \bigl(\cd_{n_{(\lb+\ell)}}\bigr).
  \]
\end{lemma}

\begin{proof}
  We first consider the case where $n(m)\ge n_{(\ell_m-1)}$. Then we can apply 
  Proposition~\ref{prop:combinatorics_of_ncrossings} (with $\lb+\ell$ in place of $\lb$) to
  show that, if $\ell_m\ge \lb+\ell+1$, the proportion of integers in $\{j\in J_m: (\td)^jx\in \cd_{n_{(\lb+\ell)}}\}$
  belonging to an $n_{(\lb + \ell)}$-crossing of size less than $h_{n_{(\lb+\ell-1)}}$ is bounded above by 
  \[
    K_1 \frac{h_{n_{(\lb+\ell-1)}}}{h_{n_{(\lb+\ell)}}} + \frac{K_2}{3^{\lb+\ell}} \le (K_1+K_2)\frac{1}{3^{\lb+\ell}}<\eta.
  \]  
  Now, if $I\subset J_m$ is an $n_{(\lb+\ell)}$-crossing with $|I|\ge h_{n_{(\lb+\ell-1)}}$, another application of 
  Proposition~\ref{prop:combinatorics_of_ncrossings} proves 
  that the proportion of integers $j\in I$ such that $(\td)^jx\in \cd_{n_{\lb}}$ is at least $(1-\eta)^{2\ell}$. We finally get in this case
  \begin{equation}
  \label{eq:lemma4.4case1}
    \gamma_{J_m}\bigl(\cd_{n_{\lb}}\bigr) \ge (1-\eta)^{2\ell+1} \gamma_{J_m} \bigl(\cd_{n_{(\lb+\ell)}}\bigr).
  \end{equation}

  Now we consider the case where $n_{(\ell_m-1)}-k(\ell_m-1) + p_1+ 2p_2 \le n(m)< n_{(\ell_m-1)}$.
  Let $n$ be the largest integer, $n\le n(m)$, such that $h_n\le |J_m|$. 
  If $n<n(m)$, then we have $h_{n+1}> |J_m|\ge \eta h_{n(m)}$. But on the other hand, $h_{n+1}<h_{n(m)}/3^{n(m)-n-1}$.
 Taking into account the definition of $p_2$ (Lemma~\ref{lemma:ac}), we get that $n(m)-n<p_2$, and finally that $n$ satisfies
 \[
   n_{(\ell_m-1)}-k(\ell_m-1) + p_1+ p_2 < n < n_{(\ell_m-1)}.
 \]

Since $|J_m|< h_ {n_{(\ell_m-1)}}$,  we observe that each coordinate can leave $C_{n_{(\ell_m-1)}}$ only
  once on $J_m$, and when it does so, it stays outside $C_{n_{(\ell_m-1)}}$ on an interval of length $\le h_{n_{(\ell_m-1)}-k(\ell_m-1)}+1$.
  Since there are $d$ coordinates, the set of integers $j\in J_m$ such that $(\td)^jx\in \cd_{n_{(\ell_m-1)}}$ is cut into at most 
  $(d+1)$ pieces, and its cardinality is at least
  \begin{align*}
    |J_m| - d\left( h_{n_{(\ell_m-1)}-k(\ell_m-1)}+1  \right) & \ge |J_m| - \dfrac{d}{3^{p_1+p_2}} h_{n} \\
    & \ge  |J_m|\left(1-\dfrac{d}{3^{p_1+p+2}}\right)\\
    & \ge  |J_m|\left(1-\dfrac{\eta}{6}\right) \ge \frac{|J_m|}{2}\\
  \end{align*}
  Therefore, there exists at least one subinterval $\tilde J_m\subset J_m$, with size 
  \[ |\tilde J_m| \ge \frac{|J_m|}{2(d+1)}\ge \eta h_{n}, \] 
  and which is contained in a single $n_{(\ell_m-1)}$-crossing. Since $n_{(\ell_m-2)}<n<n_{(\ell_m-1)}$, the estimation~\eqref{eq:lemma4.4case1} is valid for $\tilde J_m$ in place of 
  $J_m$, \textit{i.e.}
  \begin{equation}
    \label{eq:first_result}
    \gamma_{\tilde J_m}\bigl(\cd_{n_{\lb}}\bigr) \ge (1-\eta)^{2\ell+1} \gamma_{\tilde J_m} \bigl(\cd_{n_{(\lb+\ell)}}\bigr).
  \end{equation}
  Note that, if $n<n(m)$, then by definition of $n$ we have $|J_m|<h_{n+1}<7 h_n$. If $n=n(m)$, $|J_m|=h_{n(m)}=h_n$. 
  Hence in all cases we have $\eta h_{n}\le |\tilde J_m| \le |J_m| \le 7 h_n$. So we can also apply Lemma~\ref{lemma:ac} 
  with $I\egdef J_m$, $J\egdef \tilde J_m$, and $\lb+\ell$ in place of $\lb$. This yields
  \begin{equation}
    \label{eq:second_result}
    \gamma_{\tilde J_m} \bigl(\cd_{n_{(\lb+\ell)}}\bigr)\ge \theta_2(7) \gamma_{J_m} \bigl(\cd_{n_{(\lb+\ell)}}\bigr).
  \end{equation}
  Combining \eqref{eq:first_result} and  \eqref{eq:second_result}, we get
  \begin{align*}
    \Gam{J_m} &\ge \gamma_{\tilde J_m}\bigl(\cd_{n_{\lb}}\bigr) \\
    &\ge  (1-\eta)^{2\ell+1} \gamma_{\tilde J_m} \bigl(\cd_{n_{(\lb+\ell)}}\bigr)\\
    &\ge (1-\eta)^{2\ell+1}  \theta_2(7) \gamma_{J_m} \bigl(\cd_{n_{(\lb+\ell)}}\bigr).
  \end{align*}
\end{proof}

With the above lemma, we see that all the conditions needed to apply Proposition~\ref{prop:compacity} are satisfied, and this gives
the following result.

\begin{lemma}
  \label{lemma:precompacity}
  Let $(J_m)$ be a sequence of intervals satisfying \eqref{eq:assumptionJm1} and~\eqref{eq:assumptionJm2}. Then there is a subsequence $\bigl(\gamma_{J_{m_j}}\bigr)$
  which converges to some nonzero Radon measure.
\end{lemma}

\subsection{Convergence of sequences of empirical measures}

\begin{prop}
\label{prop:convergence_empirical_measures_interval}

Let $(I_m)$ and $(J_m)$ be two sequences of intervals, with $J_m\subset I_m$.
Assume that there exist two sequences of integers $(\ell_m)$ and $(n(m))$, and a real number $M>0$ such that 
\begin{itemize}
  \item $\ell_m\to\infty$,
  \item $n_{(\ell_m-1)}-k(\ell_m-1) + p_1+ 2p_2 \le n(m)< n_{\ell_m}$,
  \item $I_m$ is contained in some $n_{\ell_m}$-crossing,
  \item $\eta h_{n(m)}\le |J_m| \le h_{n(m)}$,
  \item $|I_m|\le M h_{n(m)}$.
\end{itemize}
If  $\gamma_{I_m}\tend{m}{\infty}\sigma$, we also have $\gamma_{J_m}\tend{m}{\infty}\sigma$. 
\end{prop}

\begin{proof}
The assumptions~\eqref{eq:assumptionJm1} and~\eqref{eq:assumptionJm2} are satisfied for the sequence of intervals $(J_m)$, 
hence Lemma~\ref{lemma:precompacity} applies to the sequence of measures $\big(\gamma_{J_m}\bigr)$. 
  Therefore, it is enough to prove that, if $\gamma_{J_m}$ converges to some nonzero Radon measure $\gamma$, then $\gamma=\sigma$ up to 
  some multiplicative constant. So, let us assume that $\gamma_{J_m}\to\gamma$. Since $J_m\subset I_m$, we have $\gamma_{J_m}\le \gamma_{I_m}$.
  We can also apply Lemma~\ref{lemma:ac} which shows that, for each large enough integer $\lb$, we have as soon as $n(m)>n_{\lb}+1$
  \[
    \Gam{J_m}\ge \theta_2(M)\, \Gam{I_m}.
  \]
Then Proposition~\ref{prop:absolute_continuity} ensures that $\gamma\ll\sigma$. Now by ergodicity of $(X^d,\sigma,\td)$, it 
 only remains to show that $\gamma$ is $\td$-invariant. For this we want to apply Lemma~\ref{lemma:diagonal}.
 Since $X_{\bz}^d\subset X^d\setminus X_\infty^d$, we have
 $\sigma(X_{\bz}^d)=0$ hence $\gamma(X_{\bz}^d)=0$ by absolute continuity. Finally, observe that if  for some fixed integer 
 $n$, $B$ and $B'$ are two $n$-boxes contained in the same $n$-diagonal, then for any $m$, as $J_m$ is an interval,
 \[
   \left| \gamma_{J_m}(B) - \gamma_{J_m}(B')  \right|\le 1.
 \]
 Indeed, the times $j$ when the orbit of $x$ falls in $B$ alternate with the times when the orbit of $x$ fall in $B'$.
 On the other hand the first point of Proposition~\ref{prop:combinatorics_of_ncrossings} ensures that 
 \[
   \gamma_{J_m}\bigl(\cd_n\bigr)\tend{m}{\infty}\infty, 
 \]
 and it follows that $\gamma(B)=\gamma(B')$. Lemma~\ref{lemma:diagonal} now proves that $\gamma$ is $\td$-invariant.
\end{proof}

\begin{remark}
  Note that the condition $\gamma_{I_m}\tend{m}{\infty}\sigma$ is automatically satisfied if $0\in I_m$ for each $m$, since
  we took $x$ as a typical point for $\sigma$.
\end{remark}

Now we want to extend Proposition~\ref{prop:convergence_empirical_measures_interval} to the case where $(J_m)$ is no longer 
a sequence of intervals, but $J_m$ is a subset of $I_m$ with a sufficiently regular structure.

\begin{prop}
  \label{prop:convergence_empirical_measures_pieces}
  Let $(I_m)$ and $(J_m)$ be two sequences of finite subsets of $\ZZ$, and let $M>0$. We assume that the following conditions are satisfied.
  \begin{itemize}
    \item $J_m\subset I_m$ for each $m$.
    \item There exists a sequence of integers $(\ell_m)$ with $\ell_m\to\infty$ as $m\to\infty$, such that for each $m$, $I_m$ is
    an interval  contained in some $n_{\ell_m}$-crossing.
    \item There exists a sequence of integers $(n(m))$ with
    \[
      n_{(\ell_m-1)} - k(\ell_m-1)+p_1+2p_2 \le n(m) < n_{\ell_m},
    \]
    such that for each $m$, $J_m$ is a disjoint union of intervals of common size $s(m)$, where $\eta h_{n(m)}\le s(m) \le h_{n(m)}$, 
    and $I_m\setminus J_m$ does not contain an interval of size greater than $Mh_{n(m)}$.    
  \end{itemize}

  If  $\gamma_{I_m}\tend{m}{\infty}\sigma$, then we also have $\gamma_{J_m}\tend{m}{\infty}\sigma$. 
\end{prop}

\begin{proof}
  We just have to justify that the same arguments as in the proof of Proposition~\ref{prop:convergence_empirical_measures_interval}
  apply also in this case. First, we want to prove that the conclusion of Lemma~\ref{lemma:precompacity} holds for $(J_m)$. For this,
  it is enough to observe that all the pieces of $J_m$ satisfy assumptions~\eqref{eq:assumptionJm1} and~\eqref{eq:assumptionJm2}.
  Hence the estimation given in Lemma~\ref{lemma:ratio} is valid for each piece of $J_m$, and then it is also valid for $J_m$ itself. 
  
  Now, let $\lb$ be a large enough integer, and take $m$ large enough so that $n(m)>n_{\lb}+1$.
  Let $J$ be any piece of $J_m$ (in particular we have $|J|=s(m)\ge \eta h_{n(m)}$), and let $I$ be the interval constituted of $J$ and the two 
  adjacent pieces of $I_m\setminus J_m$. Then we have $|I|\le 40 h_{n(m)}+s(m) \le 100  h_{n(m)}$, and we can apply Lemma~\ref{lemma:ac}
  to $I$ and $J$ to get 
  \[
    \Gam{J} \ge \theta_2 \Gam{I}.
  \]
  Summing over all pieces $J$ of $J_m$, we get 
  \[
    \Gam{J_m} \ge \theta_2 \Gam{I_m},
  \]
  which is the second key step in the proof of Proposition~\ref{prop:convergence_empirical_measures_interval}.
  This ensures that, if $\gamma_{J_m}\to\gamma$, then $\gamma\ll\sigma$.
  
  Finally, we have to see that $\gamma$ is $\td$-invariant, and it is enough for that to show that, if for some fixed $n$ we consider
  two $n$-boxes $B$ and $B'$ on the same $n$-diagonal, then $\gamma(B)=\gamma(B')$. But for each $m$ and each piece $J$ of 
  $J_m$, we have $\left| \gamma_J(B) - \gamma_J(B') \right| \le 1$, whereas by the first point of Proposition~\ref{prop:combinatorics_of_ncrossings},
  we know that 
  \[
    \min_{J\text{ piece of }J_m} \Gam{J} \tend{m}{\infty} \infty.
  \]

\end{proof}

\section{Twisting transformations and decomposition of $\sigma$ as a product}

\label{sec:twist}

The purpose of this section is to provide a criterion ensuring that $\sigma$ can be decomposed into the 
product of two measures $\sigma_1\times\sigma_2$, $\sigma_i$ being 
a $T^{\times d_i}$-invariant Radon measure on $X^{d_i}$ for some $1\le d_i<d$, $d_1+d_2=d$.
We will need for that to introduce the following type of transformation of $X^d$.
\begin{definition}
\label{def:twist}  The transformation $S:X^d\to X^d$ is said to be a \emph{twisting transformation} if there exists a partition 
  $\{1,\ldots,d\}=G_0\sqcup G_1$ into two nonempty subsets such that
  for each $(y_1,\ldots,y_d)\in X^d$,
  \[
      S(y_1,\ldots, y_d)= (z_1,\ldots,z_d),\text{ where } z_i\egdef\begin{cases}
                                                                 T y_i &\text{ if }i\in G_1, \\
                                                                 y_i &\text{ if }i\in G_0.
                                                               \end{cases}
   \]
\end{definition}

The reason why we introduce those twisting transformations is that, if we are able to prove that $\sigma$ is invariant by some 
twisting transformation then $\sigma$ can be decomposed as a product of two measures. More precisely, by Theorem~A.1 in~\cite{ChaconInfinite}
we have the following result.
\begin{prop}
  \label{prop:product}
  Assume that $\sigma$ is invariant by some twisting transformation $S$, and let $\{1,\ldots,d\}=G_0\sqcup G_1$ be the partition associated 
  with $S$. Then there exist Radon measures $\sigma_0$ and $\sigma_1$ on $X^{G_0}$ and $X^{G_1}$ respectively, such that 
  \begin{itemize}
    \item $\sigma=\sigma_0\otimes\sigma_1$;
    \item each $\sigma_s$ is $T^{\times |G_a|}$-invariant ($a=0,1$), and the system $(X^{G_a},T^{\times |G_a|},\sigma_a)$ is conservative ergodic.
  \end{itemize}
\end{prop}

Thus, if the assumption of the above proposition is satisfied, we can write $\sigma$ as the product of two measures which are invariant by
some smaller Cartesian power of $T$, and to which we can apply the induction hypothesis to finish the proof of Theorem~\ref{thm:msj}.
We want now to give a condition under which we are able to prove that $\sigma$ is indeed invariant by some twisting transformation.

In the next proposition, we use again the notation $\ovC_n$ which was introduced in the proof of Lemma~\ref{lemma:diagonal}: recall 
that $\ovC_n$ is the union of all $n$-boxes of the form $L_n^{j_1}\times\cdots\times L_n^{j_d}$ where for all $i=1,\ldots,d$, 
$j_i\neq0$. We observe that, if $B$ is an $n$-box contained in $\ovC_n$ and if $S$ is a twisting transformation, then $S^{-1}B$ is 
also an $n$-box (but not necessarily contained in $\ovC_n$). Note also that $\ovC_n\subset \ovC_{n+1}$ for each $n$.

\begin{prop}
  \label{prop:twist}
  Assume that there exist $(\sigma'_n)$, $(\sigma_n)$, two sequences of Radon measures on
  $X^d$, and a sequence $(S_n)$ of twisting transformations satisfying
  \begin{itemize}
    \item $\sigma_n \tend{n}{\infty} \sigma$,
    \item $\sigma'_n \tend{n}{\infty} \sigma$,
    \item For each $m\ge0$, and for $n$ large enough (depending on $m$), for each $m$-box $B\subset\ovC_m$, 
    $\sigma'_n(S_n^{-1}B)=\sigma_n(B)$.
  \end{itemize}
    Then there exists a twisting transformation $S$ such that $\sigma$ is $S$-invariant. In particular, $\sigma$ is a product measure as in
    Proposition~\ref{prop:product}.
\end{prop}

\begin{proof}
  Note that, $d$ being fixed here, there exist only finitely many twisting transformations. Therefore, considering subsequences if necessary, we 
  may assume that there exist a twisting transformation $S$ such that $S_n=S$ for each $n$.
  
  Now, let $m\ge0$ be large enough so that $\delta(m)<1/2$ (see Corollary~\ref{cor:edge}). In particular, $\sigma(\cd_m)>0$. Let $m'\ge m$. Then, for each $n$ large enough (depending on $m'$), if $B$ is an $m'$-box contained in $\ovC_m$, then $B\subset\ovC_{m'}$
  and we know that
  $\sigma'_n(S^{-1}B)=\sigma_n(B)$. By Remark~\ref{rem:mn}, the assumptions of the lemma also yield
\[
    \dfrac{\sigma'_n(S^{-1}B)}{\sigma'_n(\cd_m)}\tend{n}{\infty} \dfrac{\sigma(S^{-1}B)}{\sigma(\cd_m)}.
\]
But the left-hand side of the above formula is equal to 
  \[
    \dfrac{\sigma'_n(S^{-1}B)}{\sigma'_n(\cd_m)}=\dfrac{\sigma_n(B)}{\sigma_n(\cd_m)} \dfrac{\sigma_n(\cd_m)}{\sigma'_n(\cd_m)},
  \]
  where
  \[
    \dfrac{\sigma_n(B)}{\sigma_n(\cd_m)}\tend{n}{\infty} \dfrac{\sigma(B)}{\sigma(\cd_m)}.
  \]
  It remains to control the ratio  $\sigma_n(\cd_m)/{\sigma'_n(\cd_m)}$. For this, we write
  \[
    \sigma'_n(\cd_m)=\sigma'_n(S^{-1}\ovC_m)+\sigma'_n(\cd_m\setminus S^{-1}\ovC_m),
  \]
  and 
  \[
    \sigma_n(\cd_m)=\sigma_n(\ovC_m)+\sigma_n(\cd_m\setminus \ovC_m).
  \]
  We observe that the first terms $\sigma'_n(S^{-1}\ovC_m)$ and $\sigma_n(\ovC_m)$ are equal.
  Moreover, as $\cd_m\setminus \ovC_m$ and $\cd_m\setminus S^{-1}\ovC_m$ are included in $\partial \cd_m$,
  we have by Corollary~\ref{cor:edge}
  \[
    \frac{\sigma_n(\ovC_m)}{\sigma_n(\cd_m)}\tend{n}{\infty}  \frac{\sigma(\ovC_m)}{\sigma(\cd_m)} \ge 1-\delta(m)
  \]
and  
\[
    \frac{\sigma'_n(S^{-1}\ovC_m)}{\sigma'_n(\cd_m)}\tend{n}{\infty}  \frac{\sigma(S^{-1}\ovC_m)}{\sigma(\cd_m)} \ge 1-\delta(m),
  \]
  where $\delta(m)\to 0$ as $n\to \infty$. In particular, $\sigma_n(\ovC_m)$ and $\sigma'_n(S^{-1}\ovC_m)$ are positive
  if $\delta(m)<1/2$ and $n$ is large enough. Hence we can write
  \[
    \dfrac{\sigma_n(\cd_m)}{\sigma'_n(\cd_m)} = \dfrac{1+ \frac{\sigma_n(\cd_m\setminus \ovC_m)}{\sigma_n(\ovC_m)} }{1+ \frac{\sigma'_n(\cd_m\setminus S^{-1}\ovC_m)}{\sigma'_n(S^{-1}\ovC_m)} }
    \tend{n}{\infty} \dfrac{1+ \frac{\sigma(\cd_m\setminus \ovC_m)}{\sigma(\ovC_m)} }{1+ \frac{\sigma(\cd_m\setminus S^{-1}\ovC_m)}{\sigma(S^{-1}\ovC_m)} }.
  \]
  This yields
  \[
    \sigma(S^{-1}B)={\sigma(B)} 
    \dfrac{1+ \frac{\sigma(\cd_m\setminus \ovC_m)}{\sigma(\ovC_m)} }{1+ \frac{\sigma(\cd_m\setminus S^{-1}\ovC_m)}{\sigma(S^{-1}\ovC_m)} }.
  \]
  But the above argument is also valid if, at the beginning, we start with $m'$ instead of $m$ (and keep the same $m'$-box $B$.) This gives 
  \[
    \sigma(S^{-1}B)={\sigma(B)} 
    \dfrac{1+ \frac{\sigma(\cd_{m'}\setminus \ovC_{m'})}{\sigma(\ovC_{m'})} }{1+ \frac{\sigma(\cd_{m'}\setminus S^{-1}\ovC_{m'})}{\sigma(S^{-1}\ovC_{m'})} }.
  \]
  Moreover, as $\sigma(\cd_m)>0$, we can choose the $m'$-box $B$ in such a way that $\sigma(B)>0$, and comparing the last two equalities, we get 
  \[
     \dfrac{1+ \frac{\sigma(\cd_m\setminus \ovC_m)}{\sigma(\ovC_m)} }{1+ \frac{\sigma(\cd_m\setminus S^{-1}\ovC_m)}{\sigma(S^{-1}\ovC_m)} } =
    \dfrac{1+ \frac{\sigma(\cd_{m'}\setminus \ovC_{m'})}{\sigma(\ovC_{m'})} }{1+ \frac{\sigma(\cd_{m'}\setminus S^{-1}\ovC_{m'})}{\sigma(S^{-1}\ovC_{m'})} }.
  \]
  But the ratio on the right-hand side can be made arbitrarily close to 1 by choosing $m'$ large enough, hence it is equal to 1. This proves that, for any $m'\ge m$ and any $m'$-box $B\subset \ovC_m$, 
  $\sigma(S^{-1}B)=\sigma(B)$. We thus get as in the proof of Lemma~\ref{lemma:diagonal} that $\sigma$ and $S_*(\sigma)$ coincide on
  $\bigcup_m \ovC_m=X\setminus   X_{\bz}^d$. And since both measures are equal to 0 on $X_{\bz}^d$, this concludes the proof.  
\end{proof}

The first two assumptions in Proposition~\ref{prop:twist} will be given by applications of Proposition~\ref{prop:convergence_empirical_measures_interval} and  Proposition~\ref{prop:convergence_empirical_measures_pieces}. The following simple example presents the main ideas of how to construct sequences of measures $(\sigma_n)$ and $(\sigma'_n)$ satisfying the third requirement of Proposition~\ref{prop:twist}.

\begin{example}
\label{Ex:twist}
  Let $n \notin \{n_\ell: \ell\ge1\}$. Let $J$ be an interval contained in an $n$-crossing, set $J'\egdef J+h_n$,  assume that $J'$ is also contained in an $n$-crossing. Finally, assume that, for each 
  $j\in J$, $\bigl\{ t_n(T^jx_i): i=1,\ldots,d\bigr\}=\{1,2\}$.
  
  Define $\sigma_n\egdef \gamma_J$, and 
  $\sigma'_n\egdef\gamma_{J'}$. For an arbitrary $j\in J$, consider the partition $\{1,\ldots,d\}=G_0\sqcup G_1$ into two nonempty subsets, where 
  $G_0\egdef\bigl\{i\in\{1,\ldots,d\}:  t_n(T^jx_i)=1\bigr\}$, and $G_1\egdef\bigl\{i\in\{1,\ldots,d\}:  t_n(T^jx_i)=2\bigr\}$. Note that, since $J$ is contained in an $n$-crossing, this partition does not depend on the choice of $j\in J$. Let $S_n$ be the twisting transformation associated with $(G_0,G_1)$. Then for each $m\le n$ and each $m$-box $B\subset \ovC_m$, we have $\sigma'_n(S_n^{-1}B)=\sigma_n(B)$.
  
  Indeed, consider first $i\in G_0$. Then, for $j\in J$, $T^j(x_i)$ is in the first subcolumn of tower~$n$, hence when the orbit of $x_i$ reaches the top of tower~$n$, it will see no spacer before coming back to $C_n$. We thus have $j_n(T^{j+h_n}x_i)=j_n(T^jx_i)$, and in fact we have the equality $j_m(T^{j+h_n}x_i)=j_m(T^jx_i)$ for each $m\le n$ (remember that $j_n$ determines $j_m$ for $m\le n$).  
  
  On the other hand, if $i\in G_1$, the orbit of $x_i$ will pass through the spacer above the middle subcolumn before coming back to $C_n$, and we have, for each $m\le n$, $j_m(T^{j+h_n}x_i)=j_m(T^jx_i)-1$ (provided $j_m(T^jx_i)\neq 0$). 
  
  Now, if $B\subset\ovC_m$ is an $m$-box for some $m\le n$, the above argument shows that, for each $j\in J$, $(\td)^jx\in B\Longleftrightarrow (\td)^{j+h_n}x\in S^{-1}B$.
\end{example}

\section{End of the proof of the main result}
\label{sec:end}
Now we come back to the last part of the proof of Theorem~\ref{thm:msj}. 
We interpret Proposition~\ref{prop:synchronized} as follows: 
\begin{itemize}
    \item either $\sigma$ is a graph measure arising from powers of $T$,
    \item or there exist infinitely many integers $n$ such that there exists at least one substantial $n$-crossing which is not synchronized.
\end{itemize}
It only remains to show how this latter property implies that $\sigma$ can be decomposed as a 
product measure, as explained in Section~\ref{sec:main} and with the tools of Section~\ref{sec:twist}.

From now on, we thus assume that for infinitely many integers $n$, there exists at least one substantial $n$-crossing 
which is not synchronized. We have to study different cases, according to the relative positions of these integers $n$ with respect to the sequence $(n_\ell)$.

\subsection{The case $n_{(\ell-1)} \le n \le n_{\ell}-\ell$}
\label{sec:easycase}
Here we first assume that there exist infinitely many integers $n$ for which 
\begin{itemize}
  \item there exists at least one substantial $n$-crossing 
which is not synchronized. 
\item $\exists\ell: n_{(\ell-1)} \le n \le n_{\ell}-\ell$.
\end{itemize}
 
 Let us consider such an $n$. To unify the treatments of the cases $n=n_{(\ell-1)}$ and $n_{(\ell-1)} < n \le n_{\ell}-\ell$, 
 we set 
 \[
   \tilde h_n \egdef \begin{cases}
                       h_{n}+h_{n(\ell-1)-k(\ell-1)} &\text{ if }n=n_{(\ell-1)},\\
                       h_n  &\text{ if }n_{(\ell-1)} < n \le n_{\ell}-\ell.
                     \end{cases}
 \]
 In this way, as long as we stay inside the interval $[-100 h_n, 100h_n]$ (which is contained in a single $n_\ell$-crossing as $n \le n_{\ell}-\ell$),
 if $j$ is in some $n$-crossing and $j_n(T^jx_i)>0$, we have
\[
  j_n\left(T^{j+\tilde h_n }x_i\right) = \begin{cases}
                                       j_n\left(T^{j}x_i\right)  &\text{ if }t_n(T^jx_{i})=1,\\
                                       j_n\left(T^{j}x_i\right) -1    &\text{ if }t_n(T^jx_{i})=2,\\
                                       \text{one or other of the above values}     &\text{ if }t_n(T^jx_{i})=3.
                                     \end{cases}
\]
Let $J$ be a substantial $n$-crossing which is not synchronized. Then for $j\in J$, $\{t_n(T^j x_i):i=1,\ldots,d\}$ contains at least two different values (which do 
not depend on the choice of $j\in J$ since $j\mapsto t_n(T^jx_i)$ is constant on an $n$-crossing).

We first assume that $\{1,2\}\subset \{t_n(T^j x_i):i=1,\ldots,d\}$. Then, by the above formula,  for $j\in J\setminus \min J$, the difference 
$j_n\left(T^{j}x_i\right) - j_n\left(T^{j+\tilde h_n }x_i\right)$ takes both values 0 and 1 as $i$ runs over $\{1,\ldots,d\}$.
Set, for $a=0,1$
\[
  G_a\egdef \left\{i: j_n\left(T^{j}x_i\right) - j_n\left(T^{j+\tilde h_n }x_i\right) = a\right\}.
\]
Then we can define a twisting transformation $S_n$ with this partition. We also define the interval $J'\egdef J+\tilde h_n $, and
the two measures $\sigma_n\egdef \gamma_{J}$, $\sigma'_n\egdef \gamma_{J'}$.

As explained in Example~\ref{Ex:twist}, if for some $B$ is an $m$-box for some $m\le n$  with 
$B\subset \ovC_m$, we then have
\begin{equation}
  \label{eq:S_n-invariance}\sigma_n'(B)=\sigma_n(S_n^{-1}(B)).
\end{equation}

Let us explain how we construct $S_n$, $\sigma_n$ and $\sigma'_n$ when $\{2,3\}= \{t_n(T^j x_i):i=1,\ldots,d\}$ for $j\in J$.
Then, for $j\in J\setminus\{\max J\}$, the difference 
$j_n\left(T^{j-\tilde h_n}x_i\right) -  j_n\left(T^{j}x_i\right)$ takes both values 0 and 1 as $i$ runs over $\{1,\ldots,d\}$. 
In this case we define the partition by 
\[
  G_a\egdef \left\{i: j_n\left(T^{j-\tilde h_n}x_i\right) -  j_n\left(T^{j}x_i\right) = a\right\}, \quad a=0,1,
\]
and the corresponding twisting transformation $S_n$. We
consider $J'\egdef J-\tilde h_n$, $\sigma_n\egdef \gamma_{J'}$ and $\sigma'_n\egdef \gamma_{J}$, 
and we get~\eqref{eq:S_n-invariance} for any $m$-box $B\subset \ovC_m$, $m\le n$.

Finally we consider the case when $\{1,3\}= \{t_n(T^j x_i):i=1,\ldots,d\}$ for $j\in J$.
Then, for $j\in  J\setminus\{\min J\}$, there are two options: 
                            \begin{itemize}
                              \item either there exists $i\in\{1,\ldots,d\}$ with $t_n(T^jx_i)=3$, and $j_n\left(T^{j}x_i\right) - j_n\left(T^{j+\tilde h_n }x_i\right)=1$
                              (we see one spacer above the third column for at least one coordinate),
                              \item or for each $i\in\{1,\ldots,d\}$  such that $t_n(T^jx_i)=3$, we have $j_n\left(T^{j}x_i\right) - j_n\left(T^{j+\tilde h_n }x_i\right)=0$
                              (we see no spacer above the third column).
                            \end{itemize}
In the first option, we do the same construction as in the case $\{1,2\}\subset \{t_n(T^j x_i):i=1,\ldots,d\}$.
In the second option, we observe that 
\[
  j_n\left(T^{j}x_i\right) - j_n\left(T^{j+2\tilde h_n }x_i\right) 
  = \begin{cases}
      1 & \text{ if }t_n(T^jx_i)=1,\\
      0 & \text{ if }t_n(T^jx_i)=3.\\      
    \end{cases}
\]
We then set $J'\egdef J+2\tilde h_n $, and construct $S_n$, $\sigma_n$ and $\sigma'_n$ as before.

Since we assume that there are infinitely many integers $n$ with these properties, we can apply Proposition~\ref{prop:convergence_empirical_measures_interval}
to prove that $\sigma_n\to\sigma$ and $\sigma'_n\to\sigma$. Indeed, $J$ and $J'$ are both contained in $\{-5 h_n,\ldots,5 h_n\}$ which is contained 
in an $n_\ell$-crossing. Since $J$ is a substantial $n$-crossing, we have $\eta h_n\le |J|=|J'| \le h_n$, and we have $\gamma_{\{-5 h_n,\ldots,5 h_n\}}\tend{n}{\infty}\sigma$.
Then Proposition~\ref{prop:twist} shows that $\sigma$ can be decomposed as a product of two Radon measures to which we can apply the induction hypothesis.

\smallskip

We are now reduced to study the case where, for each $\ell$ large enough and each $n_{\ell-1}\le n \le n_\ell-\ell$, all substantial $n$-crossings are synchronized, but still
there exist infinitely many integers $n$ for which at least one substantial $n$-crossing is not synchronized.

\subsection{The case $n_\ell-\ell < n<n_\ell-k(\ell)$}
\label{sec:moredifficult}

This case cannot be treated as the preceding one since, for such an $n$, we are not sure any more that an interval around 0 and of size of order $h_n$
is completely contained in an $n_\ell$-crossing. Hence on such an interval, when the orbit of some $x_i$ leaves $C_n$, it may stay out of $C_n$ for a long time
(up to $h_{n_\ell-k(\ell)}+1$, which may be much larger than $h_n$).

The following lemma is introduced to remedy this problem. 
\begin{lemma}
  \label{lemma:n_good}
  For each large enough $\ell$, there exists an integer \[\ngood(\ell)\in \{n_\ell - k(\ell)+ p_1+2p_2,\ldots,n_\ell - k(\ell)+ p_1+2p_2+d\}\] such that
  $\{h_{\ngood(\ell)},\ldots,2h_{\ngood(\ell)}\}$ is contained in an $n_\ell$-crossing.  
\end{lemma}

\begin{proof}
 Assume that $\ell>\max_i \ell(x_i)$, and that $n_\ell - k(\ell)+ p_1+2p_2+d<n_\ell$.
  We say that the coordinate $i\in\{1,\ldots,d\}$ is \emph{bad for} $n$ if 
  there exists some $j\in \{h_n,\ldots,2h_n\}$ such that $T^j x_i\notin C_{n_\ell}$. We observe that if $\{h_n,\ldots,2h_n\}$ is not
  contained in an $n_\ell$-crossing, then at least one coordinate is bad for $n$. To prove the lemma, it is sufficient to show that for each 
  $i=1,\ldots,d$, there is at most one $n\in \{n_\ell - k(\ell)+ p_1+2p_2,\ldots,n_\ell - k(\ell)+ p_1+2p_2+d\}$ for which $i$ is bad. 
  So assume that $i$ is bad for some $n$ in this interval, and let $j\in \{h_n,\ldots,2h_n\}$ such that $T^j x_i\notin C_{n_\ell}$. 
  The orbit of $x_i$ comes back to $C_{n_\ell}$ before $j+h_{n_\ell-k(\ell)}+1$, then stays in $C_{n_\ell}$ on an interval of length $h_{n_\ell}$.
  But we have $j+h_{n_\ell-k(\ell)}+1<h_{n+1}$ and $j+h_{n_\ell}>2 h_{n_\ell - k(\ell)+ p_1+2p_2+d}$, hence $i$ cannot be bad for any $n'>n$ in the interval 
  $\{n_\ell - k(\ell)+ p_1+2p_2,\ldots,n_\ell - k(\ell)+ p_1+2p_2+d\}$.
\end{proof}

\begin{remark}
\label{rem:6.2}
  It follows from Proposition~\ref{prop:convergence_empirical_measures_interval} that we have the following 
  convergence:
  \[
    \gamma_{\{h_{\ngood(\ell)},\ldots,2h_{\ngood(\ell)}\}} \tend{\ell}{\infty} \sigma.
  \]
\end{remark}
Indeed, this proposition applies where $\ell$ plays the role of $\ell_m-1$, $\ngood(\ell)$ is $n(m)$, $\{h_{\ngood(\ell)},\ldots,2h_{\ngood(\ell)}\}$ is $J_m$,  and 
$\{0,\ldots,2h_{\ngood(\ell)}\}$ is $I_m$.

\smallskip

We will also need the following result, which will also be useful in the next section.
We consider here an integer $n$ such that $n_\ell-\ell< n \le n_\ell-k(\ell)$ for some $\ell$, and we set $n'\egdef n_\ell-k(\ell)$. 
As in the proof of Lemma~\ref{lemma:consecutive_n_intervals} we introduce the fake $n'$-tower, and the fake $n$-tower that 
mimicks the structure of tower~$n$ inside tower $n'$. (Note that this is possible as long as $n\le n'$.)
$\tilde C_n$ is the union of the levels of the fake $n$-tower, and $\overline{C}_n\egdef C_n\sqcup\tilde C_n$.
Recall that $j_n$ indicates the level of tower~$n$ to which a point in $C_n$ belongs. We extend this definition to points in $\overline{C}_n$:
$\overline{j}_n$ indicates the level of tower~$n$ (possibly fake) to which a point in $\overline{C}_n$ belongs.

\begin{lemma}
\label{lemma:shifting}
  For each large enough $\ell$, for each $n$ such that $n_\ell-\ell< n \le n_\ell-k(\ell)$, for each integer $r$ such that $|rh_n|<10 h_{n_\ell}$,
  for each $i=1,\ldots,d$, we have
  \begin{itemize}
    \item $x_i\in C_n$ and $4^\ell<j_n(x_i)<h_n-1-4^\ell$,
    \item $T^{rh_n}x_i\in \overline{C}_n$,
    \item $j_n(x_i)-4^\ell\le\overline{j}_n \left( T^{rh_n} x_i\right)\le j_n(x_i)+4^\ell$.
  \end{itemize}
\end{lemma}

\begin{proof}
  If $\ell-1\ge \max_i \ell(x_i)$ (\textit{cf. Lemma~\ref{lemma:Xinfty}}), we have $x_i\in C_{n_{(\ell-2)}} \subset C_n$ for $i=1,\ldots,d$.
  Moreover, $x_i$ is not in the first hundred occurrences of tower~$n_{(\ell-1)}-(\ell-1)$ inside tower $n_{(\ell-1)}$. Hence, as $n_\ell/\ell\to\infty$ as $\ell\to\infty$,
  and remembering~\eqref{eq:s_n},
  we have for $\ell$ large enough
  \[
    j_n(x_i) \ge j_{n_{(\ell-1)}}(x_i) \ge 100 h_{n_{(\ell-1)}-(\ell-1)} > 100\times 3^{n_{(\ell-1)}-(\ell-1)} > 100\times 3^{n_{(\ell-2)}} > 4^\ell.
  \]
  By a symmetric argument, we also get for $\ell$ large enough $j_n(x_i)<h_n-1-4^\ell$.

 We observe that, since $|rh_n|<10 h_{n_\ell}$, $\{-|rh_n|,\ldots,|rh_n|\}$ is contained in an $n_{(\ell+1)}$-crossing. Hence when the orbit of some coordinate leaves $C_n$
  on this interval, it comes back after 0, 1, $h_{n'}$ or $h_{n'}+1$ iterations of the transformation. If we consider the enlarged tower $\overline{C}_n$ instead of $C_n$, then 
  $T^jx_i$ comes back to $\overline{C}_n$ after 0 or 1 iteration of the transformation. Hence $\overline{j}_n \left( T^{h_n} x_i \right)\in \{j_n(x_i)-1, j_n(x_i)\}$, 
  and by a simple induction we get  $\overline{j}_n \left( T^{rh_n} x_i \right)\in \{j_n(x_i)-|r|,\ldots, j_n(x_i)+|r|\}$.
  The result then follows from the fact that $|r|<4^{\ell}$ (indeed, by hypothesis we have $n>n_\ell-\ell$, hence $h_n>10 h_{n_\ell}/4^\ell$ for $\ell$ large enough).  
\end{proof}

\begin{remark}
\label{remark:fake-t_n}
If, as in the case we are currently studying, we have the strict inequality $n< n_\ell-k(\ell)$, then the number of occurrences of the fake $n$-tower 
inside the fake $n'$-tower is a multiple of 3.   So we can extend the function $t_n$ 
  to a function $\overline{t}_n$ defined on $\overline{C}_n$ in such a way that, for each $r$ such that $|rh_n|<h_{n_\ell}$ and each $i=1,\ldots,d$
  \[
    \overline{t}_n \left( T^{rh_n} x_i\right) =  \overline{t}_n (x_i) + r \bmod 3.
  \]
\end{remark}

We consider now an integer $n$ with $n_\ell-\ell<n<n_\ell-k(\ell)$ for some $\ell$, where $\ell$ is large enough to apply the preceding lemmas,
and we assume that there is at least one substantial $n$-crossing which is not synchronized.
With the assumption stated at the end of Section~\ref{sec:easycase}, we can also assume that for each $n_{(\ell-1)}\le m\le n-1$, all substantial $m$-crossings
are synchronized. Then, as in the second part of the proof of Proposition~\ref{prop:synchronized}, we can construct inductively a family $(J_m)_{n_{(\ell-1)}\le m\le n}$
where
\begin{itemize}
  \item $J_{n_{(\ell-1)}}$ is a substantial $n$-crossing of length $\ge (1-(d+2)\eta)h_{n_{(\ell-1)}}$;
  \item for each $m>n_{(\ell-1)}$, $J_m$ is a substantial $m$-crossing extending $J_{m-1}$ and of size $|J_m|\ge h_m-(d+2)\eta h_{n_{(\ell-1)}}$ (see Remark~\ref{remark:sizes}).
\end{itemize}
In particular, the size of the $n$-crossing $J_n$ satisfies $|J_n|\ge h_n-(d+2)\eta h_{n_{(\ell-1)}}$, and we can assume that $\ell$ is large enough so that this implies
$|J_n|\ge (1-\eta/100)h_n$. Since we assume that there is at least one substantial $n$-crossing which is not synchronized, 
this ensures that $J_n$ itself is not synchronized.
Indeed, assume that there is another substantial $n$-crossing $J'_n$ which is not synchronized. Because the length of $J_n$ is so close to $h_n$, the orbit of each coordinate has to pass through the top of tower~$n$ between $J_n$ and the other substantial $n$-crossing $J'_n$. But $J'_n$ intersects $\ci_n$, hence the distance between $J_n$ and $J'_n$ is less than $h_n$. This shows that, for each $i=1,\ldots,d$,  $t_n(x_i)$ increases by 1 $\bmod$ 3 between the two substantial $n$-crossings. Then, as $J'_n$ is not synchronized, $J_n$ itself is not synchronized.

Moreover, by Remark~\ref{remark:1.6}, the $(n_\ell-\ell)$-crossing containing 0 covers the interval $\{-100  h_{n_{(\ell-1)}},\ldots, 100  h_{n_{(\ell-1)}}\}$. In particular, it contains
$J_{(\ell-1)}$ hence it is $J_{(n_\ell-\ell)}$. As $J_n$ extends $J_{(n_\ell-\ell)}$, this proves that $J_n$ contains 0. 

Consider the set
\[
  R\egdef \Bigl\{r\ge1: J_n+rh_n\subset\{h_{\ngood(\ell)},\ldots,2h_{\ngood(\ell)}\} \Bigr\}
\]
Since $2h_{\ngood(\ell)}<h_{n_\ell}$, Lemma~\ref{lemma:shifting} applies to each $r\in R$. In particular, for each $r\in R$ and each $i=1,\ldots,d$
we have $T^{rh_n}x_i\in\overline{C}_n$. But by choice of $\ngood(\ell)$, we also know that $T^{rh_n}x_i\in C_{n_\ell}$. Since $C_{n_\ell}$ is disjoint
from the fake $n'$-tower, $T^{rh_n}x_i\notin \tilde C_n$, and finally $T^{rh_n}x_i\in C_n$. Let $\tilde J_n$ be the interval obtained by removing the first
$4^\ell$ elements of the $n$-crossing $J_n$. Then, by Lemma~\ref{lemma:shifting}, we have $0\in\tilde J_n$, and for each $r\in R$, $\tilde J_n+rh_n$ is
contained in an $n$-crossing. Note that the size of $\tilde J_n$ is $\ge (1-\eta/100)h_n-4^\ell>(1-\eta) h_n$.

By Remark~\ref{remark:fake-t_n}, for each $r\in R$ and each $i=1,\ldots,d$, we have $t_n(T^{rh_n}x_i)=t_n(x_i)+r\bmod 3$. In particular, as $J_n$ is
not synchronized, for each $r\in R$, $t_n(T^{rh_n}x_i)$ takes at least 2 values as $i$ varies. 
We define 
\[
  r_0:=\min \{r\in R: t_n(T^{rh_n}x_i)\text{ takes both values 1 and 2 as }i=1,\ldots,d\}.
\]
We have $\min R\le r_0\le \min R+2$.
%

Now let us consider $r$ such that both $r$ and $r+1$ are in $R$. 
For $j\in\tilde J_n$, we want to compare the position in tower~$n$ of $T^{j+rh_n}x_i$ and $T^{j+(r+1)h_n}x_i$ for each coordinate.
\begin{itemize}
  \item If $i$ is such that $t_n(T^{rh_n}x_i)=1$, the orbit of $x_i$ will not pass through a spacer between 
$\tilde J_n+rh_n$ and  $\tilde J_n+(r+1)h_n$. Hence in this case we have $j_n(T^{j+rh_n}x_i)-j_n(T^{j+(r+1)h_n}x_i)=0$.
  \item If $i$ is such that $t_n(T^{rh_n}x_i)=2$, the orbit of $x_i$ will  pass through one spacer between 
$\tilde J_n+rh_n$ and  $\tilde J_n+(r+1)h_n$. Hence in this case we have $j_n(T^{j+rh_n}x_i)-j_n(T^{j+(r+1)h_n}x_i)=1$.
  \item If $i$ is such that $t_n(T^{rh_n}x_i)=3$, we have $j_n(T^{j+rh_n}x_i)-j_n(T^{j+(r+1)h_n}x_i)\in\{0,1\}$, 
  depending on the position of $T^{rh_n}x_i$ in the subsequent towers.
\end{itemize}
More precisely,  in every case the value of $j_n(T^{j+rh_n}x_i)-j_n(T^{j+(r+1)h_n}x_i)$ is determined as follows:
let $m$ be the smallest integer, $m\ge 0$, such that $t_{n+m}(T^{rh_n}x_i)\neq 3$. Note that $n+m<n_\ell$ since  $\tilde J_n+rh_n$ and $\tilde J_n+(r+1)h_n$
are contained in the same $n_\ell$-crossing. Then we have
\begin{equation}
  \label{eq:how_to_decide}
  j_n(T^{j+rh_n}x_i)-j_n(T^{j+(r+1)h_n}x_i)=\begin{cases}
                                              0 & \text{ if }t_{n+m}(T^{rh_n}x_i)=0,\\
                                              1 & \text{ if }t_{n+m}(T^{rh_n}x_i)=1.
                                            \end{cases}
\end{equation}

The difficulty which arises here is that, when $t_n(x_i)=3$, the value of this difference may vary with $r$.
This is why we need the following lemma.

\begin{lemma}
  \label{lemma:sameS}
  There exists an integer $s$, $0\le s<3^{d-1}$, such that
  \begin{itemize}
    \item $s=0\bmod 3$,
    \item for each $i=1,\ldots,d$, there exists a smaller integer $m_i$, $0\le m_i\le d-2$, satisfying
    $t_{n+m_i}(T^{(r_0+s)h_n}x_i)\neq 3$.
  \end{itemize}

\end{lemma}
\begin{proof}
We first remark that for each $i=1,\ldots,d$ and each $m\ge0$, the map $r\in R\mapsto t_{n+m}(T^{rh_n}x_i)$ 
has a very regular behaviour. Indeed, it is constant on intervals of length $3^m$, and if both $r$ and $r+3^m$ are in
$R$, we have 
\begin{equation}
  \label{eq:t_m}
  t_{n+m}(T^{(r+3^m)h_n}x_i) = t_{n+m}(T^{rh_n}x_i) + 1 \bmod 3.
\end{equation}

If $\bigl\{i\in\{1,\ldots,d\}:t_n(T^{r_0h_n}x_i)=3\bigr\}=\emptyset$, we just have to set $s\egdef 0$ and we get the 
result with $m_i=0$ for each $i$. Otherwise, we consider 
\[
  i_1 \egdef \min \bigl\{i\in\{1,\ldots,d\}:t_n(T^{r_0h_n}x_i)=3 \bigr\}.
\]
Then we choose $s_1\in\{0,1,2\}$ such that $t_{n+1}(T^{(r_0+3s_1)h_n}x_{i_1})=1$, which is possible by~\eqref{eq:t_m}.
We note that replacing $r_0$ by $(r_0+3s_1)$ does not affect the values of the $t_n(T^{rh_n}x_i)$.
Now, if $\bigl\{i\in\{1,\ldots,d\}:t_{n+1}(T^{(r_0+3s_1)h_n}x_i)=3\bigr\}=\emptyset$, we have the result with $s=3s_1$.
Otherwise, we set
\[
  i_2 \egdef \min \bigl\{i\in\{1,\ldots,d\}:t_{n+1}(T^{(r_0+3s_1)h_n}x_i)=3 \bigr\}.
\]
(Note that $i_2>i_1$.)
Then we choose $s_2\in\{0,1,2\}$ such that 
\[
  t_{n+2}(T^{(r_0+3s_1+9s_2)h_n}x_{i_2})=1.
\]
Again, replacing $(r_0+3s_1)$ by $(r_0+3s_1+9s_2)$ does not affect the values of the $t_{n+m}(T^{rh_n}x_i)$, $m=0,1$.

We continue in this way until we have found $s_1,\ldots,s_k\in\{0,1,2\}$ such that, for each $i=1,\ldots,d$, there
exists $m$, $0\le m\le k$ such that 
\[
  t_{n+m}(T^{(r_0+3s_1+\cdots+3^ks_k)h_n}x_i)\neq3.
\]
Since the algorithm also produces an strictly increasing sequence $i_1<i_2<\cdots$ in $\{1,\ldots,d\}$, we are guaranteed
that it will stop in $k\le d$ steps. Moreover, since the sequence $i_1<\cdots<i_k$ contains no $i$ such that 
$t_n(T^{r_0h_n}x_i)\in\{1,2\}$, we have in fact $k\le d-2$. We then get the annouced result by setting 
$s\egdef 3s_1+\cdots+3^ks_k\le 3^{d-1}$.
\end{proof}

\begin{figure}[htp]
  \centering
  \includegraphics[width=12cm]{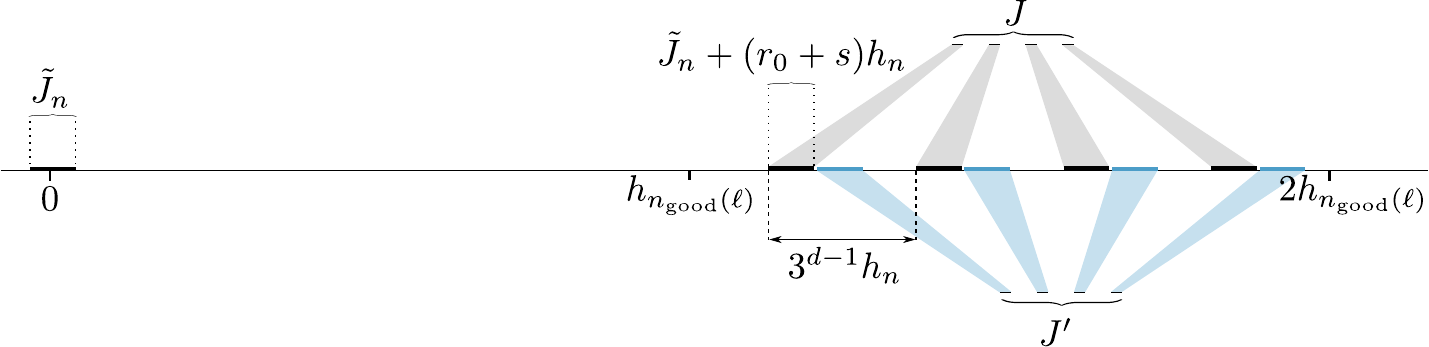}
  \caption{The choice of $J$ and $J'$ when $n_\ell-\ell < n<n_\ell-k(\ell)$}
\label{fig:cas_2}
\end{figure}

Now, with $s$ defined in Lemma~\ref{lemma:sameS}, we set 
\[
  R_1\egdef\{r\in R: (r+1)\in R\text{ and }r=r_0+s\bmod 3^{d-1}\}.
\]
Observe that $R_1\neq\emptyset$, as $R$ is an interval of size 
\[
     |R|\ge\lfloor h_{\ngood(\ell}/h_n\rfloor \ge 3^{\ngood(\ell)-n}\ge 3^{p_1} \ge 3^d.
\]
Recall that for each $0\le m\le d-2$, and each $i=1,\ldots,d$, the map $r\in R\mapsto t_{n+m}(T^{rh_n}x_i)$ is $3^{d-1}$-periodic.
Hence, by choice of $s$,  for each $i=1,\ldots,d$ the difference
\[
  j_n(T^{j+rh_n}x_i)-j_n(T^{j+(r+1)h_n}x_i)
\]
depends neither on $j\in\tilde J_n$ nor on $r\in R_1$. Moreover, by choice of $r_0$, this difference takes both values 0 and 1 as $i$ varies. 
Therefore we can construct the following partition $\{1,\ldots,d\}=G_0\sqcup G_1$, where for $a=0,1$, 
\[
  G_a\egdef \left\{i:  \forall j\in\tilde J_n,\forall r\in R_1, j_n(T^{j+rh_n}x_i)-j_n(T^{j+(r+1)h_n}x_i) = a\right\},
\]
Then we denote by $S_n$ the corresponding twisting transformation. We also consider the two disjoint subsets $J$ and $J'$ of  
$\{h_{\ngood(\ell)},\ldots,2h_{\ngood(\ell)}\}$ defined by
\[J\egdef\bigsqcup_{r\in R_1}\tilde J_n+rh_n,\quad\text{ and }J'\egdef J+h_n.\]
(See Figure~\ref{fig:cas_2})
Then, as in Example~\ref{Ex:twist}, the measures $\sigma_n\egdef \gamma_{J}$ and $\sigma'_n\egdef\gamma_{J'}$ 
satisfy~\eqref{eq:S_n-invariance} for each $m$-box 
$B\subset\ovC_m$, $m\le n$.

\smallskip

Assuming the existence of infinitely many integers $n$ with these properties, we can apply Proposition~\ref{prop:convergence_empirical_measures_pieces}
to prove that $\sigma_n\to\sigma$ and $\sigma'_n\to\sigma$. Indeed, $J$ and $J'$ are both contained in $\{h_{\ngood(\ell)},\ldots,2h_{\ngood(\ell)}\}$ 
which is contained in an $n_\ell$-crossing. They both have the structure required in the assumptions of this proposition, with  $M=3^{d-1}$. 
Moreover, we also know by Remark~\ref{rem:6.2} that  
$\gamma_{\{h_{\ngood(\ell)},\ldots,2h_{\ngood(\ell)}\}}\tend{n}{\infty}\sigma$.

Then Proposition~\ref{prop:twist} shows that $\sigma$ can be decomposed as a product of two Radon measures to which we can apply the induction hypothesis.

\smallskip

We are now reduced to study the case where, for each $\ell$ large enough and each $n_{\ell-1}\le n < n_\ell-k(\ell)$, all substantial $n$-crossings are synchronized, but still
there exist infinitely many integers $n$ for which at least one substantial $n$-crossing is not synchronized.

\subsection{The case $n = n_\ell-k(\ell)$}
\label{sec:thirdcase}

We consider now an integer $n$ of the form $n=n_\ell-k(\ell)$ for some $\ell$, where $\ell$ is large enough. We assume that there is at least one 
substantial $n$-crossing which is not synchronized, and also that for each $n_{(\ell-1)}\le m\le n-1$, all substantial $m$-crossings
are synchronized. Then, as in Section~\ref{sec:moredifficult}, we prove that the $n$-crossing $J_n$ containing 0 is of size $|J_n|\ge (1-\eta/100)h_n$, and is not synchronized. We also define $\tilde J_n\subset J_n$ as in the previous section: $\tilde J_n$ contains 0 and $|\tilde J_n|\ge (1-\eta)h_n$.

We still work with the fake tower~$n$, as introduced before Lemma~\ref{lemma:shifting} which is still valid in this case.
The new difficulty here is that we cannot anymore extend $t_n$ to $\overline{C}_n$. 

We consider integers $r$ with $0\le r\le 10d$, and we assume that $\ell$ is large enough so that $k(\ell)>10d$, thus
$10d< h_{n_\ell}/h_n$ and the results of Lemma~\ref{lemma:shifting} are valid for these integers $r$. In particular, for each such $r$, 
either $\tilde J_n+rh_n$ is contained in an $n$-crossing (we then say that $r$ corresponds to a \emph{true $n$-crossing}), 
or there is one coordinate $x_i$ such that $T^{rh_n}x_i$ is in the fake tower~$n$ $\tilde C_n$ (in this case we say that  
$r$ corresponds to a \emph{fake $n$-crossing}). Observe that for each $i=1,\ldots,d$, there is at most one integer $r$, $0\le r < \lfloor h_{n_\ell}/h_n\rfloor$, such that 
$T^{rh_n}x_i$ is in the fake tower~$n$. Indeed, as everything takes place in a single $C_{n_{(\ell+1)}}$-crossing,  when the orbit 
of $x_i$ leaves $C_{n_\ell}$, it comes back to $C_{n_\ell}$ after at most $h_n+1$ units of time, and then stays in 
$C_{n_\ell}$ for $h_{n_\ell}$ units of time.
  
If both $T^{rh_n}x_i$ and $T^{(r+1)h_n}x_i$ are in $C_n$,  then $t_n(T^{(r+1)h_n}x_i)=t_n(T^{rh_n}x_i)+1 \bmod 3$.
If $T^{rh_n}x_i$ is in the fake tower~$n$, then $T^{(r-1)h_n}x_i$ and $T^{(r+1)h_n}x_i$ are in $C_n$, and we have
$t_n(T^{(r-1)h_n}x_i)=3$, and $t_n(T^{(r+1)h_n}x_i)=1$.

With these facts in mind, we can prove the following lemma.

\begin{lemma}
  \label{lemma:thirdcase}
  There exist two consecutive integers, $-2\le r<r+1\le 10d$, such that 
  \begin{itemize}
    \item $\tilde J_n+rh_n$ is contained in an $n$-crossing,
    \item $\tilde J_n+(r+1)h_n$ is contained in an $n$-crossing,
    \item $\bigl\{i\in\{1,\ldots,d\}:t_n(T^{rh_n}x_i)=1\bigr\}\neq\emptyset,$
    \item $\bigl\{i\in\{1,\ldots,d\}:t_n(T^{rh_n}x_i)=2\bigr\}\neq\emptyset.$    
  \end{itemize}
\end{lemma}

\begin{proof}
There is at most $d$ integers $r$, $0\le r\le 10d$, such that $\tilde J_n+rh_n$ corresponds to a fake $n$-crossing (indeed, each coordinate can be 
responsible for only one $r$ for which this property fails). Hence
there is a smaller integer $r_0$, $0\le r_0\le 10d-2$, such that $r_0$, $(r_0+1)$, $(r_0+2)$ and $(r_0+3)$ correspond to true 
$n$-crossings.

If $r_0=0$, since $J_n$ is not synchronized, there are two coordinates $x_{i_1}$ and $x_{i_2}$ such that $t_n(x_{i_1})\neq t_n(x_{i_2})$. 
If $\{t_n(x_{i_1}),t_n(x_{i_2})\}=\{1,2\}$, we just have to take $r=0$. If $\{t_n(x_{i_1}),t_n(x_{i_2})\}=\{1,3\}$, we set $r=1$, 
and if $\{t_n(x_{i_1}),t_n(x_{i_2})\}=\{2,3\}$, we set $r=2$. In all these cases we get 
\[
  \{t_n(T^{rh_n}x_{i_1}),t_n(T^{rh_n}x_{i_2})\}=\{1,2\}.
\]

If $r_0>0$ and the $n$-crossing containing $r_0h_n$ is not synchronized, then we can proceed as in the previous case, replacing 0 by $r_0$.

If $r_0>0$ and the $n$-crossing containing $r_0h_n$ is synchronized, then by definition of $r_0$, $(r_0-1)$ corresponds to a 
fake $n$-crossing, hence there is at least one coordinate $x_{i_1}$ such that $t_n(T^{r_0h_n}x_{i_1})=1$. Since the corresponding
$n$-crossing is assumed to be synchronized, we have $t_n(T^{r_0h_n}x_{i})=1$ for each $i=1,\ldots,d$. We also observe
that there exists at least one coordinate $x_{i_2}$ such that $T^{(r_0-1)h_n}x_{i_2}\in C_n$. Indeed, otherwise all coordinates
would be in the fake $n$-tower at the same time, and this would imply that the $n$-crossing $J_n$ containing 0 is synchronized.
Now we take $r\egdef r_0-3$. Then for each coordinate $x_i$ such that $T^{(r_0-1)h_n}x_{i}\in C_n$, we have
$t_n(T^{rh_n}x_{i})=1$, and for each coordinate $x_i$ such that $T^{(r_0-1)h_n}x_{i}\notin C_n$, we have
$t_n(T^{rh_n}x_{i})=2$.
\end{proof}

\begin{figure}[htp]
  \centering
  \includegraphics[width=8cm]{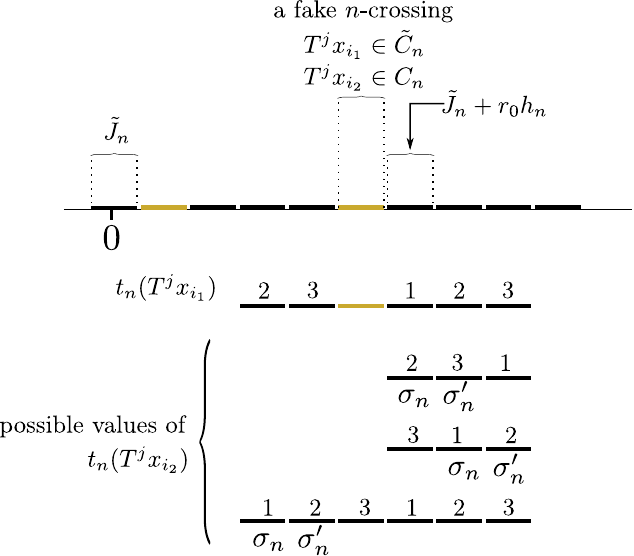}
  \caption{The three possible cases for the choice of the measures $\sigma_n$ and $\sigma'_n$ when $n=n_\ell-k(\ell)$. Here the orbit of the coordinate $x_{i_1}$ is in the fake Rokhlin tower $\tilde C_n$ on the interval $\tilde J_n+(r_0-1)h_n$, whereas on the same interval the orbit  of the coordinate $x_{i_2}$ is in $C_n$.}
\label{fig:cas_3}
\end{figure}

Now, with $r$ provided by Lemma~\ref{lemma:thirdcase}, we consider the two measures 
$\sigma_n\egdef\gamma_{\tilde J_n+rh_n}$ and $\sigma'_n\egdef\gamma_{\tilde J_n+(r+1)h_n}$ (see Figure~\ref{fig:cas_3}).
We can show by the same argument as in Section~\ref{sec:easycase} that there exists a twisting transformation 
$S_n$ such that \eqref{eq:S_n-invariance} holds whenever $B$ is an $m$-box in $\ovC_m$ for some $m\le n$. 

Finally, if we have infinitely many integers $n$ to which the above arguments apply, Proposition~\ref{prop:convergence_empirical_measures_interval}
shows that $\sigma_n\tend{n}{\infty}\sigma$ and $\sigma'_n\tend{n}{\infty}\sigma$. Then Proposition~\ref{prop:twist} shows that $\sigma$ 
can be decomposed as a product of two Radon measures to which we can apply the induction hypothesis.

\subsection{The case $n_\ell-k(\ell) < n<n_\ell$}

It only remains now to study the case where, for each $\ell$ large enough and each $n_{\ell-1}\le n \le n_\ell-k(\ell)$, all substantial 
$n$-crossings are synchronized, but still there exist infinitely many integers $n$ for which at least one substantial $n$-crossing is not synchronized. 

We consider now an integer $n$ with $n_\ell-k(\ell)<n<n_\ell$ for some large $\ell$,  
and we assume that there is at least one substantial $n$-crossing which is not synchronized.
We can also assume that for each $n_{(\ell-1)}\le m\le n-1$, all substantial $m$-crossings
are synchronized. Then, as in Section~\ref{sec:moredifficult}, we construct a family $(J_m)$ of intervals, $n_{\ell-1}\le m \le n$, 
where $J_m$ is the $m$-crossing containing 0, and is of size $|J_m|\ge h_m-(d+2)\eta h_{n_{(\ell-1)}}$.
We set $n'\egdef n_\ell-k(\ell)$. We have $|J_{n'}|\ge (1-\eta/100)h_{n'}$, provided $\ell$ is large enough.

As in Section~\ref{sec:thirdcase}, we apply Lemma~\ref{lemma:shifting} for $n'$. We consider all integers $r\ge 0$ such that 
$rh_{n'}\le 4 h_n$: each such integer $r$ corresponds either to a true $n'$-crossing (if for each $i=1,\ldots,d$, $T^{rh_{n'}}x_i\in C_{n'}$), or to a fake $n'$-crossing (if there exists $i$ such that $T^{rh_{n'}}x_i\in \tilde C_{n'}$). If $T^{rh_{n'}}x_i\in C_{n'}$, then we can consider $t_{n'}(T^{rh_{n'}}x_i)$ which evolves according to the rules stated in Section~\ref{sec:thirdcase}. We can even precise a little bit more these rules by considering also the position of $T^jx_i$ relatively to tower~$n$:
If $T^{rh_{n'}}x_i$ is in the fake tower~${n'}$, then $T^{(r-1)h_{n'}}x_i$ and $T^{(r+1)h_{n'}}x_i$ are in $C_{n'}\subset C_n$, and we have
$t_{n'}(T^{(r-1)h_{n'}}x_i)=t_n(T^{(r-1)h_{n'}}x_i)=3$, and $t_{n'}(T^{(r+1)h_{n'}}x_i)={t_n(T^{(r+1)h_{n'}}x_i)=1}$.

Let us first consider the case where each $0\le r\le \lfloor 4h_n/h_{n'}\rfloor$ corresponds to a true $n'$-crossing. 
Then the interval $\{0,\ldots,4h_n\}$ is contained in a single $n_\ell$-crossing. 
We denote by $\tilde J_n$ the interval obtained by removing from $J_n$ its first 3 points. Then, as in the proof of 
Lemma~\ref{lemma:shifting}, we prove that $0\in \tilde J_n$, and that the intervals 
$\tilde J_n$, $\tilde J_n+h_n$, $\tilde J_n+2h_n$, $\tilde J_n+3h_n$ are each contained in some $n$-crossing. 
Since $J_n$ is not synchronized, we show by similar arguments as in the proof
of Lemma~\ref{lemma:thirdcase} that for some $s\in\{0,1,2\}$, 
\[
  \{1,2\}\subset\bigl\{ t_n(T^{sh_nx_i}) : i=1,\ldots,d  \bigr\}.
\]
Then we construct the measures $\sigma_n\egdef\gamma_{\tilde J_n+sh_n}$ and $\sigma'_n \egdef\gamma_{\tilde J_n+(s+1)h_n}$: 
by similar arguments as before, we construct a twisting transformation 
$S_n$ such that \eqref{eq:S_n-invariance} holds whenever $B$ is an $m$-box in $\ovC_m$ for some $m\le n$. 
If this can be done for infinitely many integers $n$, then Proposition~\ref{prop:convergence_empirical_measures_interval}
shows that $\sigma_n\tend{n}{\infty}\sigma$ and $\sigma'_n\tend{n}{\infty}\sigma$. Then Proposition~\ref{prop:twist} shows that $\sigma$ 
can be decomposed as a product of two Radon measures to which we can apply the induction hypothesis.

Now we consider the case where there exists some $r$, $1\le r\le \lfloor 4h_n/h_{n'}\rfloor$, which corresponds to a fake $n'$-crossing. 
This case is illustrated on Figure~\ref{fig:cas_4}. We define
$r_0$ as the smallest integer with this property. Then we know that there exists $i\in\{1,\ldots,d\}$ such that $T^{r_0h_{n'}}x_i\in \tilde C_{n'}$. 
Each such $i$ is called an \emph{outgoing} coordinate. Note that for each outgoing coordinate $i$, we have for each $n'\le m\le n$ $t_m(T^{(r_0-1)h_{n'}}x_i)=3$
(indeed, the orbit of the outgoing coordinate has to reach the top of tower $m$ before leaving $C_{n_\ell}$).

\begin{figure}[htp]
  \centering
  \includegraphics[width=11cm]{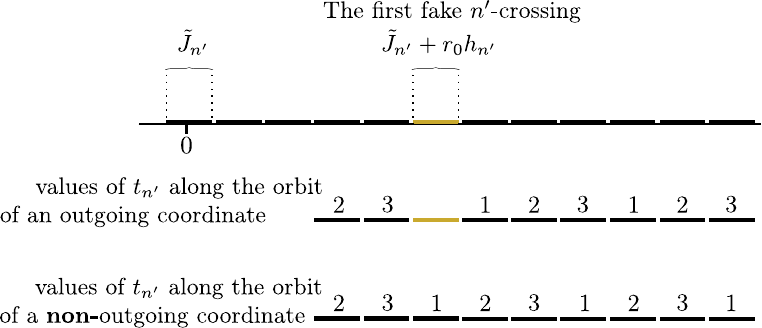}
  \caption{The behaviour of $t_{n'}$ along the orbits of outgoing and non-outgoing coordinates}
\label{fig:cas_4}
\end{figure}

We also observe that, since the interval $\{0,\ldots,(r_0-1)h_{n'}\}$ is contained in an $n_\ell$-crossing, we have for each $i_1,i_2\in\{1,\ldots,d\}$,
each $0\le r\le r_0-1$ and each $n'\le m\le n$
\begin{equation}
  \label{eq:diftm}
  t_m(T^{rh_{n'}}x_{i_1})-t_m(T^{rh_{n'}}x_{i_2})=t_m(x_{i_1})-t_m(x_{i_2}).
\end{equation}
For $n'\le m<n$, the above difference vanishes. Hence, we have $t_m(T^{(r_0-1)h_{n'}}x_{i_1})=t_m(T^{(r_0-1)h_{n'}}x_{i_2})$ for each 
$i_1,i_2$. Taking into account the outgoing coordinates, we see that for each $i=1,\ldots,d$, $t_m(T^{(r_0-1)h_{n'}}x_i)=3$. This proves that
at time $(r_0-1)h_{n'}$, each coordinate is in the last occurrence of tower $n'$ inside tower $n$. 

Now,  since the $n$-crossing $J_n$ containing 0 is not synchronized,  there exist 
$i_1,i_2$ such that the difference in~\eqref{eq:diftm} does not vanish for $m=n$, and this implies that
there exist some $i\in\{1,\ldots,d\}$ such that $t_n(T^{(r_0-1)h_{n'}}x_i)\neq3$. In particular such an $i$ is \emph{not} an outgoing coordinate. 
At time $r_0h_{n'}$, the orbit of a non-outgoing coordinate is in the first occurrence of tower $n'$ inside tower~$n$. When $r$ runs over the set 
$R\egdef\{r_0,\ldots,r_0+3^{n-n'}-1\}$,
we get that for each non-outgoing coordinate $i$,  $T^{rh_{n'}}x_i$ successively belongs to successive occurrences of tower $n'$ inside tower~$n$,
and we have $t_{n'}(T^{rh_{n'}}x_i)=r-r_0+1\bmod 3$.

On the other hand, if $i$ is an outgoing coordinate, the orbit of $x_i$ falls into the first occurrence of tower $n'$ inside tower~$n$ only at time $(r_0+1)h_n$.
And we have, for $r\in R\setminus \{r_0\}$, $t_{n'}(T^{rh_{n'}}x_i)=r-r_0\bmod 3$.

Set $R_1\egdef \{r\in R:r-r_0=1\bmod 3\}$. Then $R_1\neq\emptyset$ because $n>n'$, and for $r\in R_1$, we have 
\[
  t_{n'}(T^{rh_{n'}}x_i) = \begin{cases}
                             1 & \text{ if $i$ is an outgoing coordinate,}\\
                             2 & \text{ otherwise.}
                           \end{cases}
\]
Let $\tilde J_{n'}$ be the interval obtained after removing the first $4^\ell$ points from $J_{n'}$, and set $J\egdef \bigsqcup_{r\in R_1}\tilde J_{n'}+rh_{n'}$,
$J'\egdef J+h_{n'}$, and let $I$ be the smallest interval containing $J$ and $J'$.  Then $J$ and $J'$  have inside $I$ the structure required in  Proposition~\ref{prop:convergence_empirical_measures_pieces}, with $M=3$. We consider the two measures $\sigma_n\egdef\gamma_{J}$ and $\sigma'_n\egdef\gamma_{J'}$. 
Then there exists a twisting transformation $S_n$, defined from the partition of $\{1,\ldots,d\}$ into outgoing and non-outgoing coordinates, such that 
\eqref{eq:S_n-invariance} holds whenever $B$ is an $m$-box in $\ovC_m$ for some $m\le n'$. 

If we have infinitely many integers $n$ for which the above construction is possible, then Proposition~\ref{prop:convergence_empirical_measures_interval} ensures that $\gamma_I\tend{n}{\infty}\sigma$, then Proposition~\ref{prop:convergence_empirical_measures_pieces} yields
$\sigma_n\tend{n}{\infty}\sigma$ and $\sigma'_n\tend{n}{\infty}\sigma$. Finally, Proposition~\ref{prop:twist} shows that $\sigma$ 
can be decomposed as a product of two Radon measures to which we can apply the induction hypothesis.

\smallskip

This concludes the proof of Theorem~\ref{thm:msj}.

\section{Rational ergodicity and measurable law of large numbers}

\label{sec:rational_ergodicity}

\begin{definition}
  A \emph{measurable law of large numbers} for a conservative, ergodic, measure preserving dynamical system $(X,\A,\mu,T)$ is a measurable 
  function $L:\{0,1\}^\NN\to[0,\infty]$ such that
  for all $B\in\A$, for $\mu$-almost every $x\in X$,
  \[ 
    L\left(\ind{B}(x),\ind{B}(Tx),\ldots\right) = \mu(B).
  \]
\end{definition}

\begin{definition}
  A  conservative, ergodic, measure preserving dynamical system $(X,\A,\mu,T)$ is \emph{rationally ergodic} if there exists 
  a set $B\in\A$, $0<\mu(B)<\infty$, and a constant $M>0$ such that, for any $r\ge1$, 
  \[
    \int_B \left(\sum_{0\le j\le r-1}\ind{B}(T^jx)\right)^2 \, d\mu(x) \le M \left(\int_B \sum_{0\le j\le r-1}\ind{B}(T^jx) \, d\mu(x)\right)^2.
  \]
\end{definition}

According to Theorem~3.3.1 in~\cite{Aaronson}, a measurable law of large numbers exists for $T$ as soon as $T$ is rationally ergodic.

The purpose of this section is to prove the following result.

\begin{prop}
\label{prop:rational_ergodicity}  
The nearly finite Chacon transformation is rationally ergodic, hence admits a measurable law of large numbers.
\end{prop}

\begin{proof}
  Define $B$ as the unique level $L_0^0$ of tower~0.
  Let $r\ge h_{n_1}$ be a large integer (to be precised), and let $\ell\ge2$ be determined by the inequalities
  \[
    \frac{1}{3} h_{n_{(\ell-1)}} \le r < \frac{1}{3} h_{n_{\ell}}.
  \]
  Set $n'\egdef n_{(\ell-1)}-k(\ell-1)$. 
   We assume that $r$ is large enough so 
  that $\ell$ satisfies
  \[
    \dfrac{h_{n'}}{h_{n_{(\ell-1)}}} \le \frac{1}{3^{k(\ell-1)}} < \frac{1}{300},
   \]
  which ensures that $h_{n'}<r/100$.
  We denote by $s$ the number of levels of tower $n'$ which are contained in $B$.
  
  Let $x\in B$.  Since $B\subset C_{n'}$, 
  we have
  \[
    \sum_{0\le j\le r-1}\ind{B}(T^jx) \le \left(\frac{r}{h_{n'}}+2\right)s \le 2\frac{rs}{h_{n'}}.
  \]
  Since this is true for each $x\in B$, we get
  \begin{equation}
    \label{eq:ineqsup}
     \int_B \left(\sum_{0\le j\le r-1}\ind{B}(T^jx)\right)^2 \, d\mu(x) \le \left(2\frac{rs}{h_{n'}}\right)^2\mu(B).
  \end{equation}

 Now let us further assume that $x,Tx,\ldots,T^{r-1}x$ are all in $C_{n_\ell}$, and let us find a lower bound for the sum $\sum_{0\le j\le r-1}\ind{B}(T^jx)$. 
  Inside $\{0,\ldots,r-1\}$, we consider the disjoint subintervals of length $h_{n'}$ along which $x$ climbs into tower $n'$ or into the fake tower $n'$ (the spacers above tower 
  $n_{(\ell-1)}$). As the orbit of $x$ does not leave $C_{n_\ell}$ on $\{0,\ldots,r-1\}$, these subintervals are either contiguous or separated by one integer (corresponding to a spacer in the construction). 
  Moreover, between two successive subintervals corresponding to the fake tower $n'$, there are $3^{k(\ell-1)}$ subintervals corresponding to the real tower $n'$. Recall also that 
  $s$ is the contribution to  $\sum_{0\le j\le r-1}\ind{B}(T^jx)$ of each subinterval corresponding to a climbing into the real tower $n'$. From this it easily follows that
  \[
    \sum_{0\le j\le r-1}\ind{B}(T^jx) \ge \frac{rs}{4h_{n'}}.
  \]
  We consider the three occurrences of tower $n_\ell-1$ inside tower $n_\ell$. Observe that the proportion of points in $B$ which are in the first
  (respectively the second or the third) occurrence is exactly $1/3$. As $r< \frac{1}{3} h_{n_{\ell}}$, the orbit of a point belonging to the first 
  or the second occurrence does not leave $C_{n_\ell}$ on the interval $\{0,\ldots,r-1\}$. Hence the measure of the points in $B$ satisfying the above 
  inequality is at least $2\mu(B)/3$. 
  We get 
  \[
    \left(\int_B \sum_{0\le j\le r-1}\ind{B}(T^jx) \, d\mu(x)\right)^2 \ge  \left(\frac{rs}{4h_{n'}}\frac{2}{3}\mu(B)\right)^2.
  \]
Comparing with~\eqref{eq:ineqsup}, we get the existence of some $M$ satisfying the required inequality for any large enough $r$, which is sufficient to conclude the proof.
\end{proof}

\bibliography{nfc}

\end{document}